\documentclass[oneside,english]{siamltex}
\usepackage[T1]{fontenc}
\usepackage[latin9]{inputenc}
\usepackage{graphicx}
\usepackage{setspace}
\usepackage{amssymb}
\usepackage{esint}
\makeatletter

\newcommand{\noun}[1]{\textsc{#1}}

\newcommand\eqref[1]{(\ref{#1})}
\newenvironment{lyxlist}[1]
{\begin{list}{}
{\settowidth{\labelwidth}{#1}
 \setlength{\leftmargin}{\labelwidth}
 \addtolength{\leftmargin}{\labelsep}
 }}
{\end{list}}

\usepackage{mathrsfs}

\makeatother

\usepackage{babel}

\begin{document}

\title{Scandalously Parallelizable Mesh Generation}

\author{D. M. Bortz\thanks{Applied Mathematics, University of Colorado, Boulder, CO 80309-0526 (dmbortz@colorado.edu)} \and
A. J. Christlieb\thanks{Department of Mathematics, Michigan State University, D304 Wells Hall, East Lansing, MI 48824-1027 (christlieb@math.msu.edu).}}
\maketitle
\begin{abstract}
We propose a novel approach which employs random sampling to generate
an accurate non-uniform mesh for numerically solving Partial Differential
Equation Boundary Value Problems (PDE-BVP's). From a uniform probability
distribution $\mathcal{U}$ over a 1D domain, we sample $M$ discretizations
of size $N$ where $M\gg N$. The statistical moments of the solutions
to a given BVP on each of the $M$ ultra-sparse meshes provide insight
into identifying highly accurate non-uniform meshes. Essentially,
we use the pointwise mean and variance of the coarse-grid solutions
to construct a mapping $Q(x)$ from uniformly to non-uniformly spaced
mesh-points. The error convergence properties of the approximate solution
to the PDE-BVP on the non-uniform mesh are superior to a uniform mesh
for a certain class of BVP's. In particular, the method works well
for BVP's with locally non-smooth solutions. We present a framework
for studying the sampled sparse-mesh solutions and provide numerical
evidence for the utility of this approach as applied to a set of example
BVP's. We conclude with a discussion of how the near-perfect paralellizability
of our approach suggests that these strategies have the potential
for highly efficient utilization of massively parallel multi-core
technologies such as General Purpose Graphics Processing Units (GPGPU's).
We believe that the proposed algorithm is beyond embarrassingly parallel;
implementing it on anything but a massively multi-core architecture
would be scandalous.\end{abstract}

\section{\label{sec:Introduction}Introduction}

A wide range of linear and non-linear boundary value problems, such
as the Poisson, Helmholtz, and non-linear Poisson-Boltzmann equations,
arise in protein folding, molecular dynamics, neutral and non-neutral
plasmas, acoustics and antenna design, etc. In many of these applications,
sparse mesh-points are sufficient over vast regions of the domain,
while large, iterative, and parallel computations using a high density
of points are necessary for obtaining accuracy in the vicinity of
complex objects. In this paper, we propose a novel approach for generating
mesh-points to be used in simulating solutions to Boundary Value Problems
(BVP). The class of methods is designed to discover accurate non-uniform
discretizations via a sparse stochastic approximation.

State of the Art numerical methods typically have some mechanism for
refining the underlying mesh. At some level, all adaptive methods
are attempting to identify the optimal mesh, providing uniform accuracy
over a given domain. Adaptive methods have been heavily employed in
problems with complex geometry, where sharp geometric features can
make the simulation difficult to resolve, or in simulations where
shocks or discontinuities can arise over time.

Ideally, we would like an unstructured mesh that is designed to use
coarse resolution where the solution changes little, and fine resolution
where transitions take place, such as near a boundary. This is not
a new idea, as this is the guiding principle in deciding where to
apply Adaptive Mesh Refinement (AMR) for Hyperbolic conservation laws.
Recent work in Hyperbolic conservation laws has gone as far as developing
metrics to guide refinement which are based on an estimated error
between the numerical solution and the weak solution \cite{kk2005,kkp2002}.
Like with hyperbolic problems, the idea of working on the smallest
mesh possible while maintaining a uniform accuracy is not a new topic
when it comes to the area of numerical boundary value problems. For
boundary value problems with Lipschitz sources there are a range of
methods which can be employed to increase accuracy at boundaries where
there are transitions. One of the simplest is to make use of Chebyshev
points, which pack resolution near boundaries and still allow one
to construct high order approximate inverse operators \cite{atkinson1989}.
However, in a wide range of nonlinear boundary value problems, there
are interesting issues which arise, including the formation of interior
layers or singularities at unknown locations. In 1D, if one knew where
a transition layer or singularity is located, one could design Chebyshev
points to increase resolution nearby. Unfortunately, it may be non-trivial
to identify the transition layer or develop an efficient implementation
in higher dimensions.

Quasi Monte Carlo and low discrepancy sequences offer an alternative
approach to sparse mesh theory \cite{caflisch1998,mc1995,niederreiter1978}.
A Monte Carlo method is a statistical approach to compute the volume
of a complex domain in high dimensions. Briefly, given a complex domain,
$D_{c}$, one encompasses the domain in a simple region $D$ and then
points are sampled from $D$ and tested to see whether they lie inside
of $D_{c}$. The fraction of the points from $D$ that lie inside
$D_{c}$ are used to estimate the volume of domain $D_{c}$. As the
number of samples, $N$, go to infinity, it is well known that this
method converges at a rate of $\mathcal{O}(\sqrt{N})$. A quasi-Monte
Carlo method is formulated in the same way, except that instead of
using randomly generated points, a pre-determined set of points, called
a low-discrepancy sequence, is used in the estimation of the volume
$D_{c}$ \cite{fk2001mcqmcm}. The Quasi Monte Carlo method converges
at a rate of $\mathcal{O}(N)$. If one views a PDE from a numerical
Green's function formulation, low-discrepancy sequences are one approach
to numerically computing a solution in complex domains. These methods
have been used in the solution of integral equations \cite{cgh1998,ch2000},
and more recently, in the solution of PDE's \cite{ch2000}. The basic
premise is that a good estimate of the solution may be obtained with
relatively few samples, i.e., a sparse mesh.

Our proposed strategy differs from all the above frameworks in that
it uses sparse stochastic sampling to construct a non-uniform mesh-generation
function $Q$. The approach is ideal for identifying the placement
of mesh points to maintain uniform accuracy across the domain and
is well suited to both finite difference and finite element formulations
of a problem. Details of the approach and associated theoretical framework
are described in Section \ref{sec:Approach-and-Theoretical-Framework}
and include a preliminary investigation of convergence and well-posedness
issues. Numerical examples are then presented in Section \ref{sec:Numerical-Simulations}.
Section \ref{sec:Summary} contains a summary of our results as well
as a discussion of future directions. Lastly, since this work lies
at the intersection of numerical PDE's and statistics, some notation
may be unfamiliar to some readers and accordingly, we include a glossary
of notation in the Appendix.

\section{\label{sec:Approach-and-Theoretical-Framework}Approach and Theoretical
Framework}

Non-uniform discretizations can offer superior solution accuracy and
convergence properties. The challenge, however, lies in efficient
identification of a mesh which provides results superior to that of
uniform spacing. In this section, we offer a brief overview of our
proposed algorithm as well as the establishment of a preliminary theoretical
framework.

Consider a two-point Boundary Value Problem (BVP), e.g., the Poisson
equation \begin{eqnarray}
u''(x)=f(x)\:\mbox{ s.t.}\, & u(0)=A\,;\, & u(1)=B\,,\label{eq:poisson eq}\end{eqnarray}
for $A,B\in\mathbb{R}$ and $x\in\mathbf{I}=(0,1)$. Recall that a
uniform discretization for the second derivative yields second-order
convergence. We denote $\bar{\mathbf{I}}$ as the domain $[0,1]$,
a discretization of the interior $\mathbf{I}$ as $\mathbf{x}_{n}=(x_{1},\ldots,x_{n})^{T}$,
and the corresponding approximate solution at $\mathbf{x}_{n}$ as
$\mathbf{u}_{n}=(u_{1},\ldots,u_{n})^{T}$. Note that frequently in
the numerical analysis literature, the variables for spatial discretization
and numerical solution are denoted by a superscript $n$. Here $n$
is a subscript, which is consistent with notation for random variables
and will be used in this paper.

Next, consider a monotonically non-decreasing function $Q:\bar{\mathbf{I}}\to\bar{\mathbf{I}}$
which is absolutely continuous on a finite number of compact subsets
of $\bar{\mathbf{I}}$ and restricted at the endpoints to $Q(0)=0$,
$Q(1)=1$. The purpose of the function $Q$ is to map the uniformly
spaced mesh to a non-uniformly spaced one. We reserve $Q^{0}$ to
denote the identity mapping $Q^{0}(x)\equiv x$.

The goal is to develop a strategy for identifying a $Q$ such that,
e.g., the approximate solution to the Poisson problem \begin{eqnarray*}
u''(Q(x))=f(Q(x))\,\mbox{ s.t. }\, & u(Q(0))=A\,;\, & u(Q(1))=B\,,\end{eqnarray*}
has convergence properties (in $n$) superior to a uniform spacing.
We use a superscript $0$ to denote a uniform spacing or a variable
derived from uniform spacing, e.g., $\mathbf{x}_{n}^{0}=(\frac{1}{n+1},\ldots,\frac{n}{n+1})^{T}$,
$\mathbf{u}_{n}^{0}=(u_{1}^{0},u_{2}^{0},\ldots,u_{n}^{0})^{T}$,
$Q^{0}(\mathbf{x}_{n}^{0})\equiv\mathbf{x}_{n}^{0}$, etc. The core
of our approach is to identify $Q$ via a sparse stochastic approximation.
We repeatedly sample from a distribution $P$ and then use pointwise
statistical moments of the coarse solutions to generate the desired
non-uniform mesh function $Q$. D . Naturally, different classes of
problems call for different strategies for generating $Q$. Our results,
however, suggest that a more generalizable strategy may exist.

The rest of this section is devoted to establishing notation, framework,
and implementation details. We describe notation in Section \ref{sub:Notation}
and recall the relevant development of non-uniform finite difference
(FD) approximations in Section \ref{sub:Conventional-Non-uniform-Finite}.
In Section \ref{sub:Consistency-and-Stability}, we address how conventional
consistency and stability result can be cast in our framework. In
Section \ref{sub:Criterion-Motivation} we motivate the construction
of $Q$ as well as present the algorithms for the test cases in Section
\eqref{sec:Numerical-Simulations}. Lastly, Section \ref{sub:Sampling-Strategy}
is devoted to a description of the sampling strategy.

\subsection{\label{sub:Notation}Notation}

Let $\mathbb{X}_{n}=(X_{1},X_{2},\ldots,X_{n})$ be a random vector
where the elements are independent and identically distributed (i.i.d.)~with
$X_{k}\sim P$ for a given probability distribution $P$. The points
in the realizations must be sorted before being used to construct
an approximate solution. It is statistical convention to denote a
sorted sequence using subscripts with parenthesis. Accordingly, $\mathbb{X}_{(n)}=(X_{(1)},X_{(2)},\ldots,X_{(n)})^{T}$
denotes a \emph{sorted} random vector with elements from $\mathbb{X}_{n}$
where $X_{(k)}\sim\sum_{j=k}^{n}{n \choose j}P^{j}(1-P)^{k-j}$. The
derivation for this distribution comes from \emph{order statistics}
and we direct the interested reader to \cite{dn2003OrderStatistics}.
\begin{definition}
\label{def:U}Let $U:\bar{\mathbf{I}}^{n}\to\mathbb{R}^{n}$ be a
mapping taking a monotonically increasing sequence of $n$ distinct
points in $\bar{\mathbf{I}}$ and solving a given BVP on those points.
\end{definition}
Note that $U$ represents the action of numerically solving a PDE
and in some of the examples below, we discretize the second derivative
using a simple three-point stencil. In several cases, we implemented
higher order stencils (results not presented), and while they do provide
better convergence results, the conclusions of this work do not change.

\begin{definition}
\label{def:p}Let $p$ be a function taking a point $\xi\in\bar{\mathbf{I}}$
and a random vector of length $n$, and mapping them to a single random
variable \begin{eqnarray}
p(\xi,\mathbb{X}_{(n)}(P)) & \equiv & \mathbb{E}_{K}\left[\left\{ U(\mathbb{X}_{(n)}(P))\right\} _{K=k}\left|X_{(k)}=\xi\right.\right]\,.\label{eq:def:p}\end{eqnarray}

\end{definition}
The operator $\mathbb{E}_{K}$ denotes expectation with respect to
a uniform distribution on $\{1,\ldots,n\}$ where the distribution
of the index random variable $K$ and $\{\cdot\}_{K}$ denotes the
$K$th element of a vector. We note that this function returns a random
variable for each $\xi$.
\begin{definition}
\label{def:meanp}Let the pointwise mean of $p$ be defined for $\xi\in\bar{\mathbf{I}}$
as \begin{equation}
\mu(\xi)\equiv\mathbb{E}_{P}\left[\mathbb{E}_{K}\left[\left\{ U(\mathbb{X}_{(n)}(P))\right\} _{K=k}\left|X_{(k)}=\xi\right.\right]\right]\,.\label{eq:def:meanp}\end{equation}

\end{definition}
To appreciate the utility of this definition for $\mu$, consider
that this allows pointwise computation of the expected solution value
for an \emph{arbitrary point} $\xi\in\bar{\mathbf{I}}$ by conditioning
the expectation on $\xi$ being one of the points in a realized mesh.
The expectation on an index $K$ was constructed after much consideration
and it is productive to reflect on alternative choices. An alternative
definition of $p$ could include information (via an interpolation)
from solutions where $\xi$ is \emph{not} one of the mesh points.
We found, however, that this reduced the efficiency of our algorithm
both theoretically and practically. In addition to increasing the
complexity of the analysis, adding an interpolation step to the consistency
and stability results (Section \ref{sub:Consistency-and-Stability})
induced a reduction in the speed of solution convergence.
\begin{definition}
\label{def:varp}Let the pointwise variance of $p$ be defined for
$\xi\in\bar{\mathbf{I}}$ as \begin{equation}
v(\xi)\equiv\mathbb{V}_{P}\left[\mathbb{E}_{K}\left[\left\{ U(\mathbb{X}_{(n)}(P))\right\} _{K=k}\left|X_{(k)}=\xi\right.\right]\right]\,,\label{eq:def:varp}\end{equation}
where $\mathbb{V}_{P}$ denotes variance with respect to $P$, $\mathbb{E}_{K}$
denotes expectation with respect to $\mathcal{U}\{1,\ldots,n\}$,
the distribution of the index random variable $K$, and $\{\cdot\}_{K}$
denotes the $K$th element of a vector. We also let the average variance
over $\bar{\mathbf{I}}$ be defined as \begin{equation}
\bar{v}\equiv\frac{1}{\left\Vert \bar{\mathbf{I}}\right\Vert }\int_{\bar{\mathbf{I}}}\mathbb{V}_{P}\left[\mathbb{E}_{K}\left[\left\{ U(\mathbb{X}_{(n)}(P))\right\} _{K=k}\left|X_{(k)}=\xi\right.\right]\right]d\xi\,.\label{eq:def:varpbar}\end{equation}

\end{definition}
In the examples presented in Section \ref{sec:Numerical-Simulations},
we choose $P$ to be a uniform distribution on $\mathbf{I}$. There
is nothing in the following development, however, that would fundamentally
change for an arbitrary $P$.

\subsection{\label{sub:Conventional-Non-uniform-Finite}Conventional Non-uniform
Finite Difference}

In the example problems of Section \ref{sec:Numerical-Simulations},
we will employ finite differences. There are first and second derivatives
in these problems, which must be discretized on a nonuniform mesh.
For the first derivative, we select the simple (though unstable) centered
difference discretization. For the second derivative, however, we
create a non-uniform centered difference approximation to the differential
operator. For uniform mesh $\mathbf{x}_{n}^{0}$ and any non-uniform
mesh function $Q$, we can Taylor expand the solution $u$ at mesh-points
$k+1$ and $k-1$ around node $k$\begin{eqnarray*}
u_{k+1} & = & u_{k}+h_{k+1}u'_{k}+\frac{1}{2}h_{k-1}^{2}u''_{k}+\mathcal{O}(h_{k-1}^{3})\,,\\
u_{k-1} & = & u_{k}-h_{k}u'_{k}+\frac{1}{2}h_{k}^{2}u''_{k}-\mathcal{O}(h_{k}^{3})\,,\end{eqnarray*}
where $h_{k}=Q(x_{k}^{0})-Q(x_{k-1}^{0})$ is the step-size and $u_{k}$
is the solution evaluated at $Q(x_{k}^{0})$. This leads to a Method
of Undetermined Coefficients problem\begin{eqnarray*}
a_{k}u_{k+1}+b_{k}u_{k}+c_{k}u_{k-1} & = & (a_{k}+b_{k}+c_{k})u_{k}+(a_{k}h_{k+1}-c_{k}h_{k})u'_{k}\\
 &  & \qquad+\frac{1}{2}(a_{k}h_{k+1}^{2}+c_{k}h_{k}^{2})u''_{k}\,,\end{eqnarray*}
with the algebraic equations\begin{eqnarray*}
a_{k}+b_{k}+c_{k}=0\,,\quad & a_{k}h_{k+1}-c_{k}h_{k}=0\,,\quad & \frac{1}{2}(a_{k}h_{k+1}^{2}+c_{k}h_{k}^{2})=1\,.\end{eqnarray*}
The values of $a_{k}$, $b_{k}$, and $c_{k}$, yield the discretized
version of \eqref{eq:poisson eq}, $A_{Q(\mathbf{x}_{n}^{0})}\mathbf{u}_{n}=f_{Q(\mathbf{x}_{n}^{0})}$,
where $A_{Q(\mathbf{x}_{n}^{0})}$ is a tridiagonal matrix \begin{equation}
2\left[\begin{array}{ccccc}
-\frac{1}{h_{1}h_{2}} & \frac{1}{h_{2}(h_{1}+h_{2})} & 0 & \cdots & 0\\
\frac{1}{h_{2}(h_{2}+h_{3})} & -\frac{1}{h_{2}h_{3}} & \frac{1}{h_{3}(h_{2}+h_{3})} & \ddots & \vdots\\
0 & \frac{1}{h_{3}(h_{3}+h_{4})} & \ddots & \ddots & 0\\
\vdots & \ddots & \ddots & -\frac{1}{h_{n-1}h_{n}} & \frac{1}{h_{n}(h_{n-1}+h_{n})}\\
0 & \cdots & 0 & \frac{1}{h_{n}(h_{n}+h_{n+1})} & -\frac{1}{h_{n}h_{n+1}}\end{array}\right]\,,\label{eq:AQxn}\end{equation}
$f_{Q(\mathbf{x}_{n}^{0})}$ is the forcing function is evaluated
at the $n$ mesh-points, and $\mathbf{u}_{n}$ is the approximate
solution at the $Q(\mathbf{x}_{n}^{0})$ mesh-points. Relating this
to the notation above, we have that $\mathbf{u}_{n}=U(Q(\mathbf{x}_{n}^{0}))=A_{Q(\mathbf{x}_{n}^{0})}^{-1}f_{Q(\mathbf{x}_{n}^{0})}$.

We also note that in what follows, we consider a randomly sampled
and sorted mesh $\mathbb{X}_{(n)}(P)$ and the corresponding approximate
differential operator is denoted as $A_{\mathbb{X}_{(n)}(P)}$.

\subsection{\label{sub:Consistency-and-Stability}Consistency and Stability}

For a conventional finite difference discretization, we would consider
the error $E$ in the solution\begin{eqnarray*}
\left\Vert E(Q,\mathbf{x}_{n}^{0})\right\Vert  & = & \left\Vert u(Q(\mathbf{x}_{n}^{0}))-U(Q(\mathbf{x}_{n}^{0}))\right\Vert \\
 & = & \left\Vert A_{Q(\mathbf{x}_{n}^{0})}^{-1}\left(A_{Q(\mathbf{x}_{n}^{0})}u(Q(\mathbf{x}_{n}^{0}))-f_{Q(\mathbf{x}_{n}^{0})}\right)\right\Vert \\
 & \leq & \left\Vert A_{Q(\mathbf{x}_{n}^{0})}^{-1}\right\Vert \left\Vert \tau_{Q(\mathbf{x}_{n}^{0})}\right\Vert \,,\end{eqnarray*}
which is bounded above by the spectral radius of the inverse of the
finite difference operator $A_{Q(\mathbf{x}_{n}^{0})}^{-1}$ and a
truncation error $\tau_{Q(\mathbf{x}_{n}^{0})}$. For the non-uniform
three-point-stencil approximating the second derivative, the truncation
error is $\mathcal{O}(\max_{k}|h_{k}|)$.

For our development, however, we consider a probabilistic version
of this error, with the following conditions.

\noun{Condition C1. }\emph{For a given $P$, the spectrum of $A_{\mathbb{X}_{(n)}(P)}^{-1}$
is bounded in $[0,1]$.}

\noun{Condition C2. }\emph{For a given $P$, the truncation error
induced by a finite difference approximation to the second derivative
is first order in the largest step-size $h$.}

Moreover, we go so far as to propose that Conditions C1 and C2 are
in fact true for all classes of BVP's. In support of Condition C1,
the red circles in Figure \ref{fig:Upper-and-lower} depict the maximum
and minimum eigenvalue magnitudes for $m=12000$ realizations of $\mathbb{X}_{(n)}(P)$
for $P$ a uniform distribution on $\bar{\mathbf{I}}$. For comparison,
we have included the eigenvalue bounds for a uniform partition (blue
squares). In support of Condition C2, the red circles in Figure \ref{fig:truncation-dist}
depict (for each $n$) the maximum step-size in 1000 meshes sampled
from $P$, a uniform distribution on $\bar{\mathbf{I}}$. This illustrates
that as $n$ increases, the maximum step-size is decreasing as well,
albeit slowly in comparison with step-size in a uniform partition
(blue squares).%
\begin{figure}
\centering{}\includegraphics[width=0.45\columnwidth]{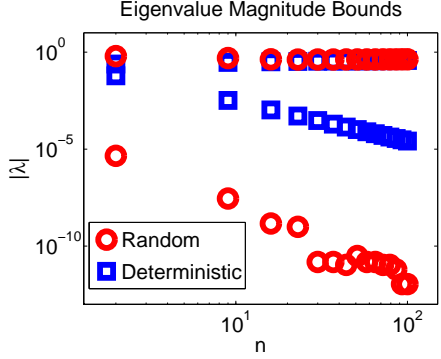}\caption{\label{fig:Upper-and-lower}Red circles are the magnitude of upper
and lower bounds for eigenvalues of $A_{\mathbb{X}_{(n)}(P)}^{-1}$.
For each $n$, $m=12000$ random grids were sampled from a uniform
distribution and eigenvalues of the corresponding $A_{\mathbb{X}_{(n)}(P)}^{-1}$
were computed. For comparison, the blue squares are the magnitudes
of the upper and lower bounds for the eigenvalues of $A_{Q^{0}(\mathbf{x}_{n}^{0})}^{-1}$.}

\end{figure}
\begin{figure}
\centering{}\includegraphics[width=0.45\columnwidth]{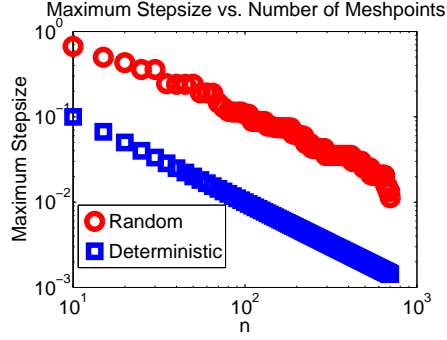}\caption{\label{fig:truncation-dist}Red circles illustrate the convergence
of the largest step-size in randomly generated meshes of length $n$.
For each $n$, $m=1000$ meshes $\mathbb{X}_{(n)}(P)$ were generated
for $P$ a uniform distribution on $\bar{\mathbf{I}}$. The vertical
axis is the maximum step-size between and two nearby points in $\mathbb{X}_{(n)}(P)$.
For comparison, we also include the step-sizes for a uniform partition
of the domain (blue squares).}

\end{figure}

\begin{theorem}
\label{thm:ptwise-conv}Under Condition C1 and C2, the expected error
converges pointwise to zero.\end{theorem}
\begin{proof}
Let $e_{k}$ be a vector of zeros with a one in the $k$th element
and $\mathbb{X}_{(n),\xi,k}$ denote the vector $(X_{(1)},\ldots,X_{(k-1)},\xi,X_{(k+1)},\ldots,X_{(n)})$
where the $k$th element of $\mathbb{X}_{(n)}$ is replaced by $\xi$.
Consider the pointwise error between the the analytical solution $u$
and mean of the coarse-mesh solutions $\mu$ \begin{eqnarray*}
\left\Vert u(\xi)-\mu(\xi;U(\mathbb{X}_{(n)}(P)))\right\Vert  & = & \left\Vert u(\xi)-\mathbb{E}_{P}\left[\mathbb{E}_{K}\left[\{U(\mathbb{X}_{(n)}(P))\}_{K}\left|K:X_{(K)}=\xi\right.\right]\right]\right\Vert \\
 & = & \left\Vert \mathbb{E}_{P}\left[\mathbb{E}_{K}\left[u(\xi)-\{U(\mathbb{X}_{(n)}(P))\}_{K}\left|K:X_{(K)}=\xi\right.\right]\right]\right\Vert \\
 & = & \left\Vert \mathbb{E}_{P}\left[\mathbb{E}_{K}\left[u(\xi)-e_{K}^{T}(U(\mathbb{X}_{(n)}(P))\left|K:X_{(K)}=\xi\right.\right]\right]\right\Vert \\
 & \leq & \mathbb{E}_{P}\left[\mathbb{E}_{K}\left[\left\Vert e_{K}^{T}(u(\mathbb{X}_{(n),\xi,K})-U(\mathbb{X}_{(n)}(P))\right\Vert \left|K:X_{(K)}=\xi\right.\right]\right]\\
 & \leq & \mathbb{E}_{P}\left[\left\Vert A_{\mathbb{X}_{(n)}(P)}^{-1}\right\Vert \cdot\mathbb{E}_{K}\left[\left\Vert A_{\mathbb{X}_{(n)}(P)}u(\mathbb{X}_{(n),\xi,K})-f_{\mathbb{X}_{(n)}}\right\Vert \left|K:X_{(K)}=\xi\right.\right]\right]\,.\end{eqnarray*}
We note that $\left\Vert A_{\mathbb{X}_{(n)}(P)}u(\mathbb{X}_{(n),\xi,k})-f_{\mathbb{X}_{(n)}}\right\Vert $
is the truncation error induced from a 3-point non-uniform centered-difference
stencil and converge as $\mathcal{O}(\max_{i=1,\ldots,n-1}|X_{(i)}-X_{(i+1)}|)$
(Figure \ref{fig:truncation-dist} also depicts this behavior). Therefore,
assuming Conditions C1 and C2, we can conclude that the error converges
pointwise.
\end{proof}

\subsection{\label{sub:Criterion-Motivation}Criteria Motivation}

In Section \ref{sec:Numerical-Simulations}, we consider problems
containing second derivatives as well as nonlinear functions of first
derivatives. We create $Q$ depending upon the specific problem class
under consideration. In one case, $Q$ is based upon the derivative
of the mean of the solutions $\mu$, while in another we consider
the product of the variance and the the second derivative of $\mu$.
In all cases, however, $Q$ is created using the statistical moments
of the sampled sparse-mesh solutions and the results suggest that
a more generalizable strategy may exist.

For the problems with second derivatives (equation \eqref{sub:Singular-Boundary-Value}
in Subsection \ref{sub:Singular-Boundary-Value} and equation \eqref{eq:leveque equation}
in Subsection \ref{sub:Oscillatory-Boundary-Value}) we define $Q$
as \[
Q(x)=\left[\frac{q_{1}(\cdot)}{q_{1}(1)}\right]^{-1}(x)\,,\]
where \[
q_{1}(x)=\int_{0}^{x}\sqrt{\left|\mu'(\xi;U(\mathbb{X}_{(n)}(P))\right|}d\xi\,,\]
and the superscript $-1$ is an inverse function operator. Essentially,
this definition will pile up points in regions with a steep solution
in an effort to provide higher order accuracy for the nonuniform second
derivative discretization.

For the problem with a third power of the first derivative (equation
\eqref{eq:HamJac eqn} in Subsection \ref{sub:Hamilton-Jacobi-Equation}),
we define $Q$ as \[
Q(x)=\left[\frac{q_{2}(\cdot)}{q_{2}(1)}\right]^{-1}(x)\,,\]
where \[
q_{2}(x)=\int_{0}^{x}\mu''(\xi;U(\mathbb{X}_{(n)}(P))^{2}\, v(\xi;U(\mathbb{X}_{(n)}(P))^{3}d\xi\,,\]
and $v$ is from Definition \eqref{def:varp}.

To motivate these mappings, consider a least squares cost function
$J:\mathbb{R}\to\mathbb{R}$ which compares (pointwise) a proposed
solution to the actual solution $u$ at an arbitrary point $\xi\in\bar{\mathbf{I}}$.
We then consider the random variable $J(\mu(\xi,\mathbb{X}_{(n)}(P))$
and Taylor series expand the expected value of $J$ around a local
minimum $\mu^{*}(\xi)$ for some $\xi\in\mathbf{I}$ \[
\mathbb{E}_{P}\left[J(\mu(\xi,\mathbb{X}_{(n)}(P))\right]\approx\mathbb{E}_{P}\left[J(\mu^{*}(\xi))+J_{1}+J_{2}\right]\]
where \begin{eqnarray*}
J_{1} & = & J'(\mu^{*}(\xi))(\mu(\xi,\mathbb{X}_{(n)}(P))-\bar{\mu}^{*}(\xi))\end{eqnarray*}
and \begin{eqnarray*}
J_{2} & = & J''(\mu^{*}(\xi))\left(\mu(\xi,\mathbb{X}_{(n)}(P))-\mu^{*}(\xi)\right)^{2}\,.\end{eqnarray*}
Note that since $\mu^{*}$ is assumed to be a local minimum for $J$
, $J_{1}\equiv0$ and that in $J_{2}$, $J''(\mu^{*}(\xi))=1$. The
expected cost function can thus be approximated by\begin{equation}
\begin{array}{l}
\mathbb{E}_{P}\left[J(\mu(\xi,\mathbb{X}_{(n)}(P)))\right]\approx\mathbb{E}_{P}\left[J(\mu^{*}(\xi))\right]+\mathbb{E}_{P}\left[\left(\mu(\xi,\mathbb{X}_{(n)}(P))-\mu^{*}(\xi)\right)^{2}\right]\,.\end{array}\label{eq:expJ with var}\end{equation}
According to Theorem \ref{thm:ptwise-conv}, $\lim_{n\to\infty}\mathbb{E}_{P}\left[J(\mu^{*}(\xi))\right]=0$
pointwise. Most importantly, the second term in \eqref{eq:expJ with var}
is the variance of $\mu(\xi,\mathbb{X}_{(n)}(P))$. The moments of
the sparse sampled solutions, therefore, can offer insight into the
actual solution.

We hypothesize that the reason $q_{1}(x)$ works well is that the
$\mu'$ may converge faster than $\mu$. A full investigation will
require establishing the appropriate Sobolev space, though it is not
immediately clear how to proceed. 

We also hypothesize that the function $q_{2}(x)$ works well because
the second derivative (when cast as the local curvature) is inversely
proportional to the local variance of a random variable (a result
which is well known in the semi-parametric nonlinear regression literature
\cite{sw1989}). Essentially, while the $\mu''$ may not converge
quickly, the product $\mu''v$ does. We also found that multiplication
by an extra $v$ dramatically improves the computed $Q$, though an
explanation is not immediately clear. A deeper understanding of the
spectrum of $A_{\mathbb{X}_{(n)}(P)}$ and how it depends upon the
choice of $P$ will be essential to explaining the efficiency of $q_{2}(x)$.
We plan to explore both of these issues in a future paper.

\subsection{\label{sub:Sampling-Strategy}Sampling Strategy}

The proposed methodology relies on random sampling and sorting a sampled
vector from a distribution. For standard sampling problems, the confidence
intervals for variance estimation converge as $\mathcal{O}(\sqrt{m})$
for $m$ samples. In the absence of a full theoretical analysis, we
offer numerical evidence supporting our accurate computation of the
variance. Note that in all the following simulations, we use the first
example problem presented in Section \ref{sec:Numerical-Simulations},
the singular two-point BVP \eqref{eq:singularBVP eqn}-\eqref{eq:singularBVP BC}.

The relationship between the mesh size $n$ and the number of samples
$m$ is non-trivial. and Figure \ref{fig:varmnSingularBVP} illustrates
this by depicting the error in $\bar{v}$ (relative to $\bar{v}$
computed with $m=3000$ sampled from a uniform distribution on $\bar{\mathbf{I}}$)
for a range of $n$ and $m$ values. For a given $n$, though, we
do note that the error in the $\bar{v}$ computation is decreasing.
In Figure \ref{fig:neededSampesSingularBVP} we depict the number
of samples of vector size $n$ which are needed to ensure three digits
of accuracy in estimating the variance. Since the number was consistently
below $1000$ over a range of $n$, we let $m=15000$ in all subsequent
simulations (unless otherwise specified). %
\begin{figure}
\centering{}\includegraphics[width=0.45\columnwidth]{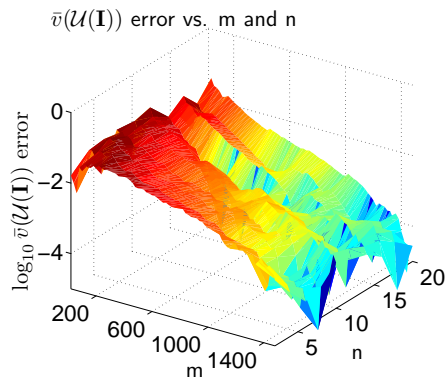}\caption{\label{fig:varmnSingularBVP}$\log_{10}$ of the error in the computation
of $\bar{v}$ (sampling from a uniform distribution on $\bar{\mathbf{I}}$)
as a function of $m$ and $n$. Note the general downward trend along
both the $m$ and $n$ axes.}

\end{figure}
\begin{figure}
\centering{}\includegraphics[width=0.45\columnwidth]{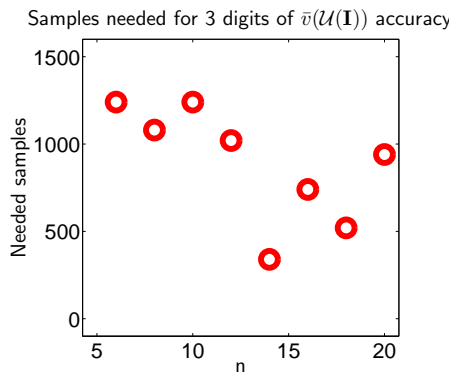}\caption{\label{fig:neededSampesSingularBVP}For each $n,$ the vertical axis
reflects the number of samples needed to compute the variance with
3 digits of accuracy relative to $\bar{v}$ (sampling from uniform
distribution on $\bar{\mathbf{I}}$) with $m=3000$.}

\end{figure}

We also note that figures with similar conclusions can be generated
for the other problems described in Section \ref{sec:Numerical-Simulations}.

\section{\label{sec:Numerical-Simulations}Numerical Simulations}

In this section, we consider three classes of BVP's and illustrate
the effectiveness of our approach in identifying a non-uniform mesh
spacing which yields a solution accuracy superior to that of uniform
spacing. The examples in Sections \ref{sub:Singular-Boundary-Value}-\ref{sub:Hamilton-Jacobi-Equation}
were chosen because the solutions exhibited changes in smoothness
such as boundary layers and a discontinuity in the derivative, i.e.,
a local effect. The example in Section \ref{sub:Oscillatory-Boundary-Value}
was chosen because the solutions exhibited varying multi-scale oscillations
over the domain, i.e., a global effect.

All software development was done in Matlab and in particular, parallel
for-loops (\textbf{parfor}) were used whenever possible. For all simulations,
we let $m=15000$ unless otherwise specified. A copy of the software
is available upon request to the corresponding author.

\subsection{\label{sub:Singular-Boundary-Value}Singular Boundary Value Problem}

The following singular two-point Boundary Value Problem\begin{eqnarray}
(x^{\alpha}u')'=f(x,u) & \, & ;\,0\leq x\leq1\label{eq:singularBVP eqn}\\
u(0)=A & \, & ;\, u(1)=B\label{eq:singularBVP BC}\end{eqnarray}
encompasses a large class of singular BVP's where $\alpha\in(0,1)$
and $A$, $B$ are finite and constant. For $(x,y)\in\{[0,1]\times\mathbb{R}\}$
such that $f(x,y)$ is continuous, $\partial f/\partial x$ exists
and is continuous and $\partial f/\partial y\geq0$. Many researchers
have successfully implemented strategies to solve classes of singular
BVP's using uniform finite difference \cite{jamet1970nummath}, Rayleigh-Ritz-Galerkin
\cite{cnv1970nummath}, projection \cite{reddien1973nummath}, and
collocation \cite{rs1976nummath} schemes.

We consider the specific problem presented in \cite{ak2001jcompappmath}
for $f(x,u)=\beta x^{\alpha+\beta-2}((\alpha+\beta-1)+\beta x^{\beta})u$,
$A=1$, $B=e$ and with exact solution $u(x)=\exp(x^{\beta})$. Depicted
in Figure \eqref{fig:sBVP-soln} is the solution for $\alpha=.85$,
and $\beta=10$; note the strong boundary layer affect near $x=1$.
\begin{figure}
\centering{}\includegraphics[width=0.45\columnwidth]{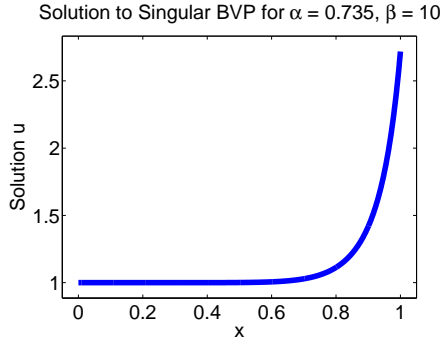}\caption{\label{fig:sBVP-soln}Analytical solution to singular two-point Boundary
Value Problem in \eqref{eq:singularBVP eqn}-\eqref{eq:singularBVP BC}
for $f(x,u)=\beta x^{\alpha+\beta-2}((\alpha+\beta-1)+\beta x^{\beta})u$,
$u(0)=1$, and $u(1)=e$. }

\end{figure}

Inspired by the fact that a non-uniform mesh provides superior accuracy
for this class of problems (as investigated in \cite{ak2001jcompappmath,ck1982nummath}),
we consider the mapping $Q_{\alpha}(x)=x^{1-\alpha}$. We choose $\alpha=0.735$,
$Q_{0.735}(x)=x^{1-0.735}$ because for $n=100$, this $Q_{0.735}$
yields 2.4 orders of magnitude more accuracy in the solution than
$Q^{0}$. The dotted curve in Figure \ref{fig:sBVP-P-comparisons}
depicts the corresponding plot of $Q_{0.735}$ and should thus be
considered a \emph{gold standard}, with which we will compare our
computed mapping.

For this problem, we first consider $m=15000$ samples of length $n=10$
and solve \eqref{eq:singularBVP eqn}-\eqref{eq:singularBVP BC}.
Figure \ref{fig:sBVP-pdf-for-unif} depicts a contour plot of the
probability density of $p(\xi;\mathbb{X}_{(n)}(P))$, where $P$ is
a uniform distribution on $\bar{\mathbf{I}}$. We note that many of
the sparse solutions have relatively large values for $x$ near $3/4$,
but we are at a loss to explain this phenomenon. Figure \ref{fig:sBVP-P-comparisons}
depicts our computed $Q$ along with the uniform and analytically
optimal mappings and Figure \ref{fig:sBVP-conv} depicts the error
convergence properties for $Q$ and $Q^{0}$. We draw attention to
the fact that for mesh sizes less than $100$, our mapping results
in second order convergence, even though the discretization is first
order. Indeed, the error is at times, marginally better than $Q_{\alpha}$.
We also note that while $Q^{0}$ and $Q_{\alpha}$ are second order
accurate, our non-uniform mapping asymptotically converges with first-order
accuracy. Lastly, Figure \ref{fig:sBVP-Solution-comparison}a depicts
solutions using $n=20$ mesh-points from the uniform and computed
optimal mappings. The uniform mesh creates a solution which is consistently
above the analytical solution, while the solution computed from $Q$
is substantially more accurate for small numbers of mesh points. While
the difference between the solution is visually small, it is almost
an order of magnitude more accurate and Figure \ref{fig:sBVP-Solution-comparison}b
depicts the solution residuals. %
\begin{figure}
\centering{}\includegraphics[width=0.5\columnwidth]{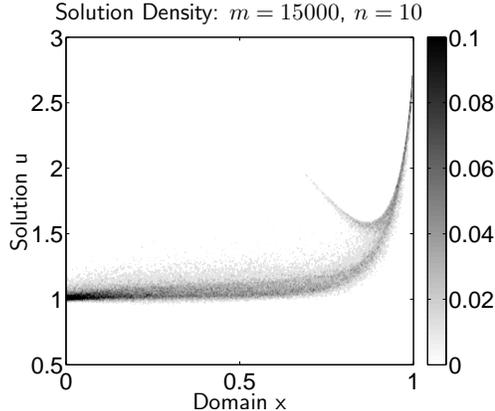}\caption{\label{fig:sBVP-pdf-for-unif}Contour plot of the probability density
of $p(x,\mathbb{X}_{(n)}(P))$ for $m=15000$ sampled meshes of length
$n=10$ from mesh mapping $Q^{0}$ and $P$ a uniform distribution
on $\bar{\mathbf{I}}$.}

\end{figure}
\begin{figure}
\centering{}\includegraphics[width=0.45\columnwidth]{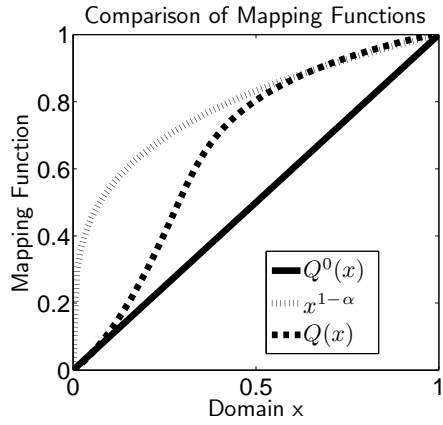}\caption{\label{fig:sBVP-P-comparisons}Comparison of different mesh mappings.
The solid line is uniform spacing $Q^{0}(x)$, the dotted line is
the analytically optimal non-uniform spacing $x^{(1-\alpha)}$, and
dot-dashed line with the circles is the computed $Q(x)$.}

\end{figure}
\begin{figure}
\centering{}\includegraphics[width=0.45\columnwidth]{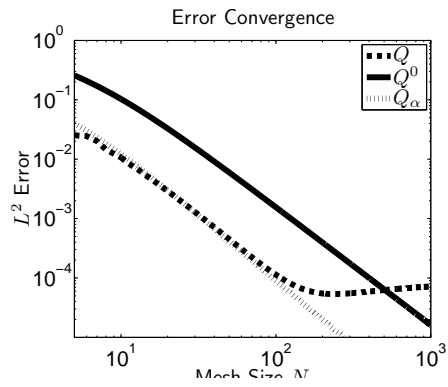}\caption{\label{fig:sBVP-conv}Error convergence of the different mesh mappings.}

\end{figure}
\begin{figure}
\centering{}a)\includegraphics[width=0.45\columnwidth]{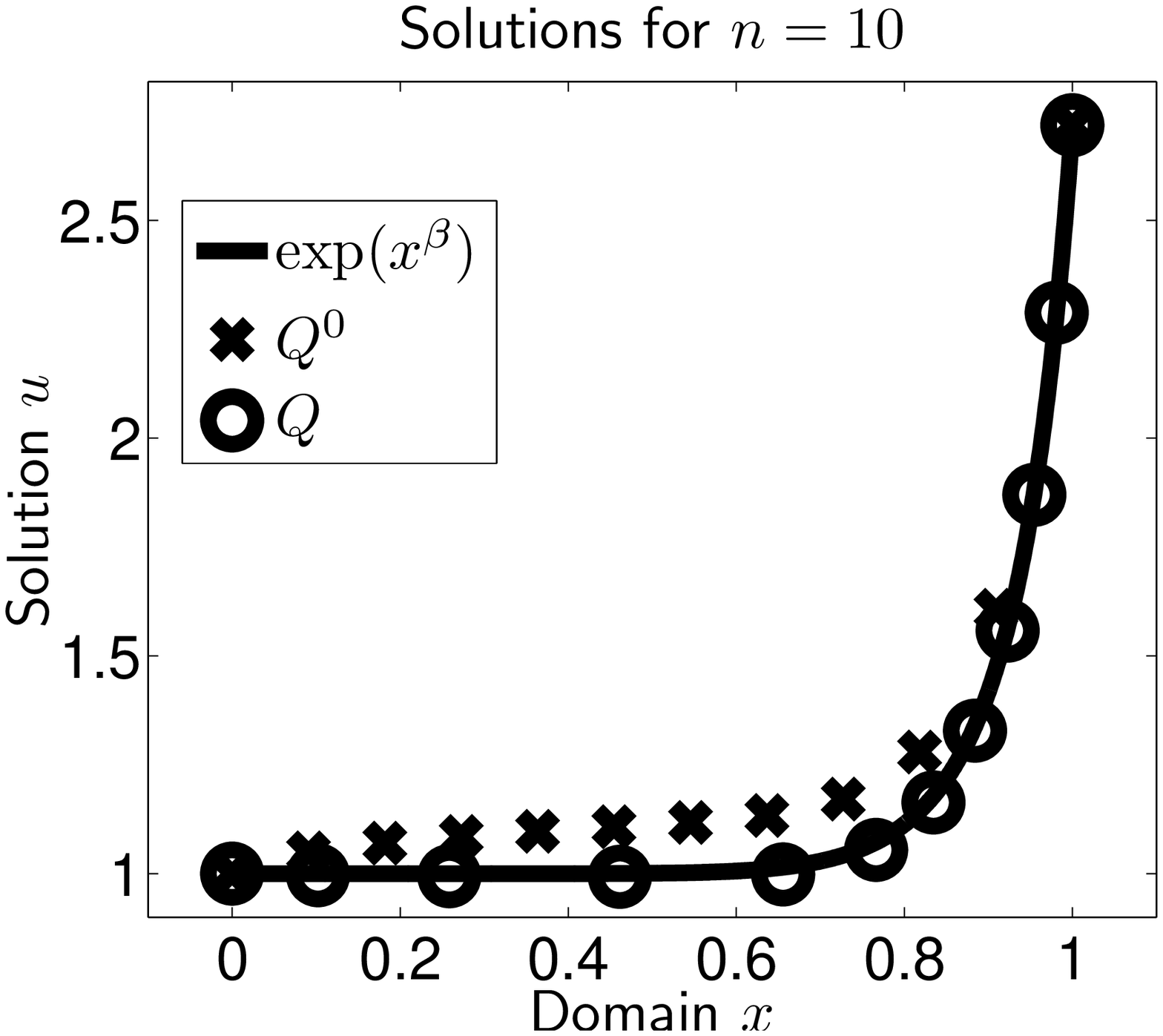}\hfill{}b)\includegraphics[width=0.45\columnwidth]{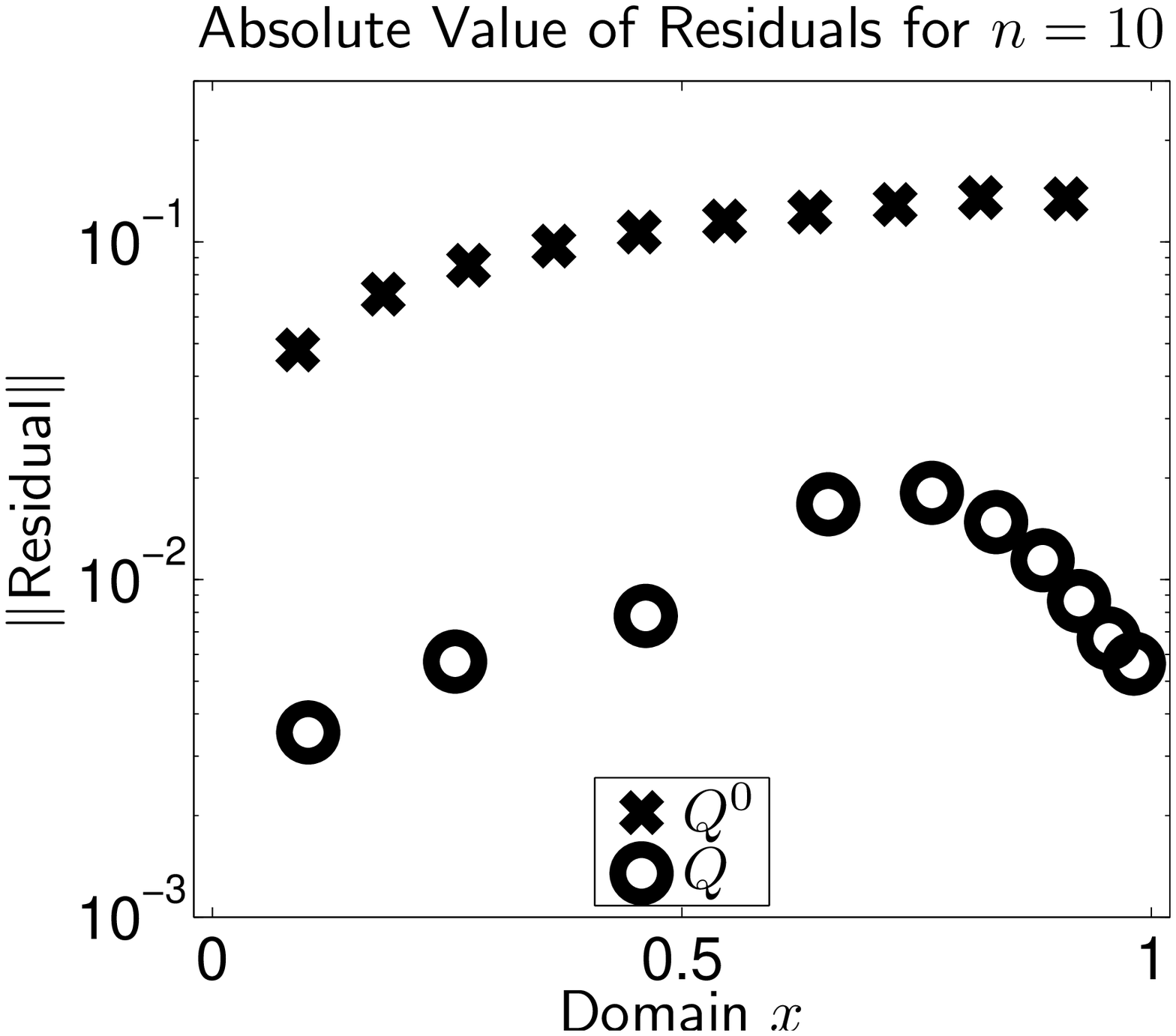}\caption{\label{fig:sBVP-Solution-comparison}a) Solution using $n=10$ mesh
points. The solid line is the analytical solution. The 'X' and 'O'
symbols are the solutions generated using the uniform ($Q^{0}$) and
computed ($Q$) mesh point mappings. b) The absolute value of the
pointwise residuals for both solutions.}

\end{figure}

\subsection{\label{sub:Hamilton-Jacobi-Equation}Hamilton-Jacobi Equation}

An important example problem which may develop a singularity at an
internal location is the time-invariant Hamilton-Jacobi (HJ) equation,\begin{eqnarray}
u+H(u_{x})=f & \, & ;\,0\leq x\leq1,\label{eq:HamJac eqn}\\
u(0) & = & u(1)\label{eq:HamJac BC}\end{eqnarray}
where we assume period Dirichlet boundary conditions. We consider
\begin{equation}
H(p)=\frac{p^{2}}{\pi^{2}}\;,\;\; f(x)=-\left|\sin(\pi(x-\frac{\pi}{4}))\right|+\cos^{2}(\pi(x-\frac{\pi}{4}))\,,\label{eq:HamJac nonsmooth f}\end{equation}
which yields a solution $u(x)=-|\sin\left(\pi\left(x-\frac{\pi}{4}\right)\right)|$
with a discontinuity in the derivative at $\frac{\pi}{4}$. The analytical
solution on $\bar{\mathbf{I}}$ is depicted in Figure \ref{fig:HamJac-soln}.
\begin{figure}
\centering{}\includegraphics[width=0.45\columnwidth]{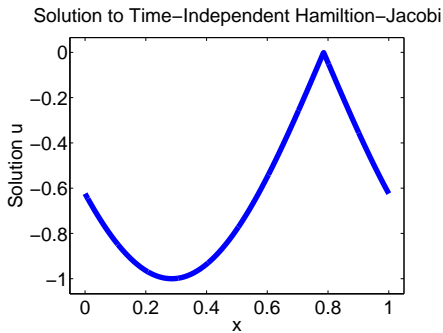}\caption{\label{fig:HamJac-soln}Analytical solution to time-independent Hamilton-Jacobi
equation in \eqref{eq:HamJac eqn}-\eqref{eq:HamJac BC} for $H(p)=\frac{p^{2}}{\pi^{2}}$
and $f(x)=-\left|\sin(\pi(x-\frac{\pi}{4}))\right|+\cos^{2}(\pi(x-\frac{\pi}{4}))$.}

\end{figure}

For this problem, we first consider $m=15000$ samples of length $n=10$
and solve \eqref{eq:HamJac eqn}-\eqref{eq:HamJac nonsmooth f}. We
discretized $u_{x}$ using centered differences and then solved this
nonlinear PDE by minimizing the norm of the residual ($u+H(u_{x})-f$)
using a Quasi-Newton method with BFGS update. Figure \ref{fig:HamJac-pdf-for-unif}
depicts the contour plot of the probability density of $p(\xi;\mathbb{X}_{(n)}(P))$
where $P$ is a uniform distribution on $\mathbf{I}$. %
\begin{figure}
\centering{}\includegraphics[width=0.5\columnwidth]{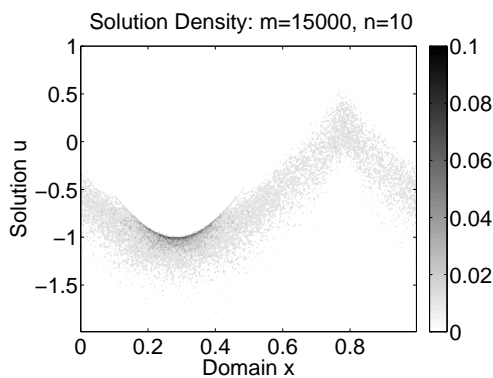}\caption{\label{fig:HamJac-pdf-for-unif}Contour plot of the probability density
of $p(x,\mathbb{X}_{(n)}(P))$ for $m=15000$ sampled meshes of length
$n=10$ for Hamilton-Jacobi problem in \eqref{eq:HamJac eqn}-\eqref{eq:HamJac nonsmooth f},
where $P$ is a uniform distribution on $\bar{\mathbf{I}}$.}

\end{figure}
\begin{figure}
\centering{}\includegraphics[width=0.5\columnwidth]{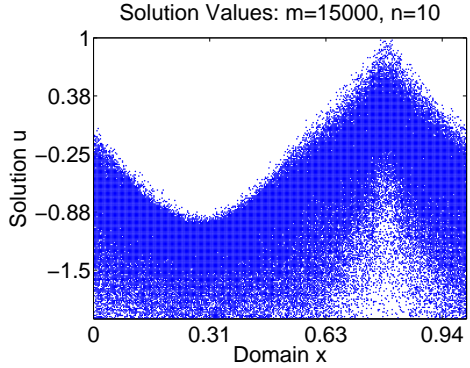}\caption{\label{fig:HamJac-point-cloud}Point cloud of solution values for
$m=15000$ sampled meshes of length $n=10$ for Hamilton-Jacobi problem
in \eqref{eq:HamJac eqn}-\eqref{eq:HamJac nonsmooth f}.}

\end{figure}

Figure \ref{fig:HamJac-comparison-P} depicts our computed optimal
$Q$ along with the uniform and analytically optimal mappings while
Figure \ref{fig:HamJac-error-conv} depicts the convergence properties
for the different mappings. We note that computed mapping will result
in many mesh-points near $\pi/4$, which is exactly the location of
the derivative discontinuity in the solution. Figure \ref{fig:HamJac-error-conv}
depicts the convergence in error for uniformly and non-uniformly spaced
points as specified by the mappings defined in Figure \ref{fig:HamJac-comparison-P}.
We note that our strategy offers a substantial improvement in solution
accuracy for small numbers of mesh-points; it even overcomes the inherent
instability of centered-differences. To illustrate the improved accuracy
for small $n$, Figure \ref{fig:HamJac-det-soln-resid} depicts a)
the solutions generated by $Q$ and $Q^{0}$ as well as b the absolute
value of the pointwise residuals for both mesh-point mappings.%
\begin{figure}
\begin{centering}
\includegraphics[width=0.45\columnwidth]{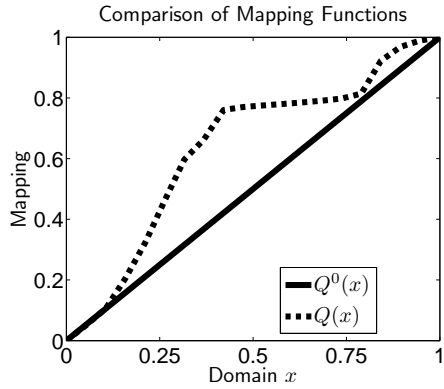}
\par\end{centering}

\caption{\label{fig:HamJac-comparison-P}Mapping functions $Q$ and $Q^{0}$.
Note how there will be a large percentage of the points at $\pi/4$.}

\end{figure}
\begin{figure}
\centering{}\includegraphics[width=0.45\columnwidth]{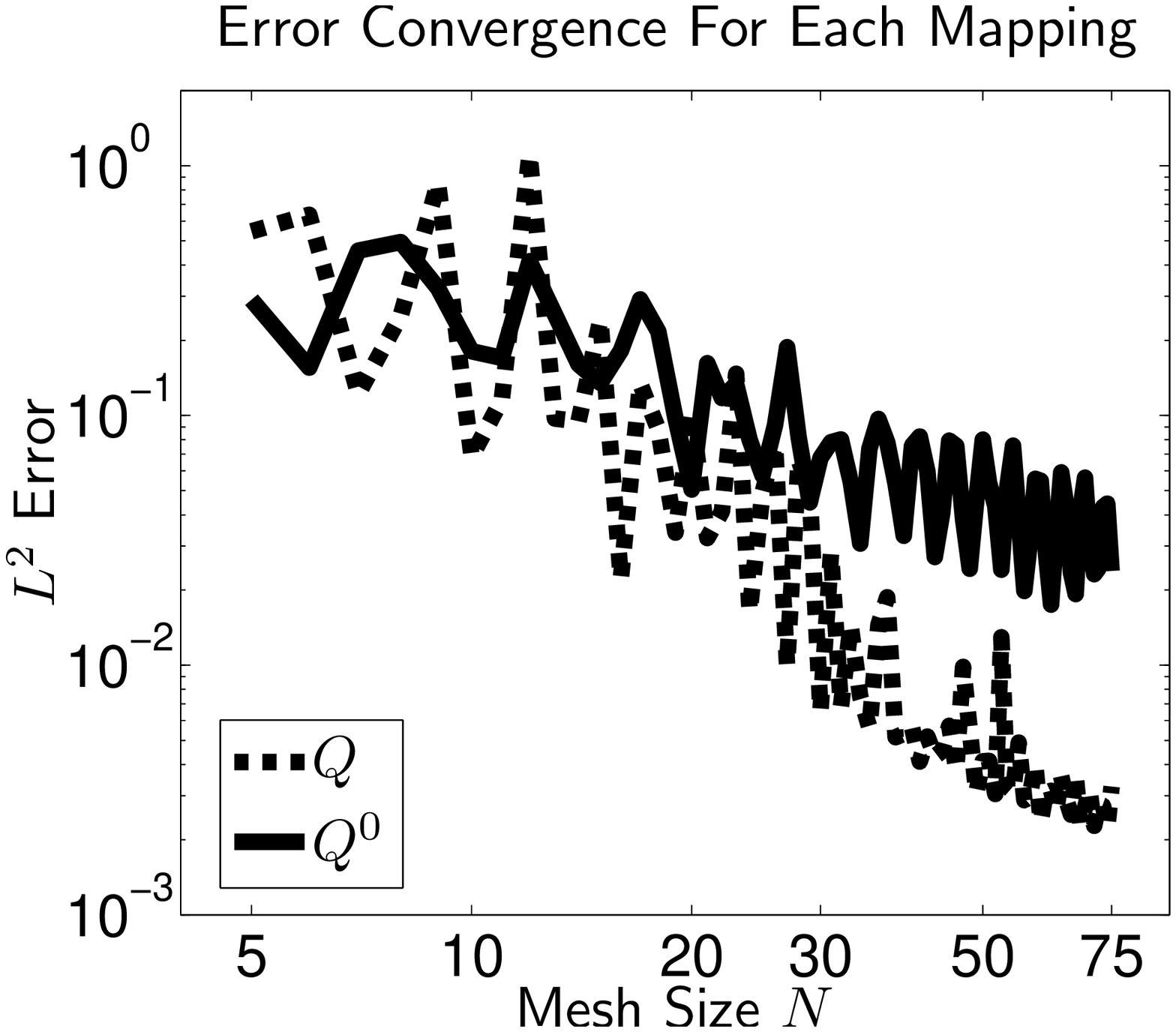}\caption{\label{fig:HamJac-error-conv}Error convergence for uniformly and
non-uniformly spaced points where the non-uniform spacing is determined
by the mapping in Figure \ref{fig:HamJac-comparison-P}.}

\end{figure}
\begin{figure}
\centering{}a)\includegraphics[width=0.45\columnwidth]{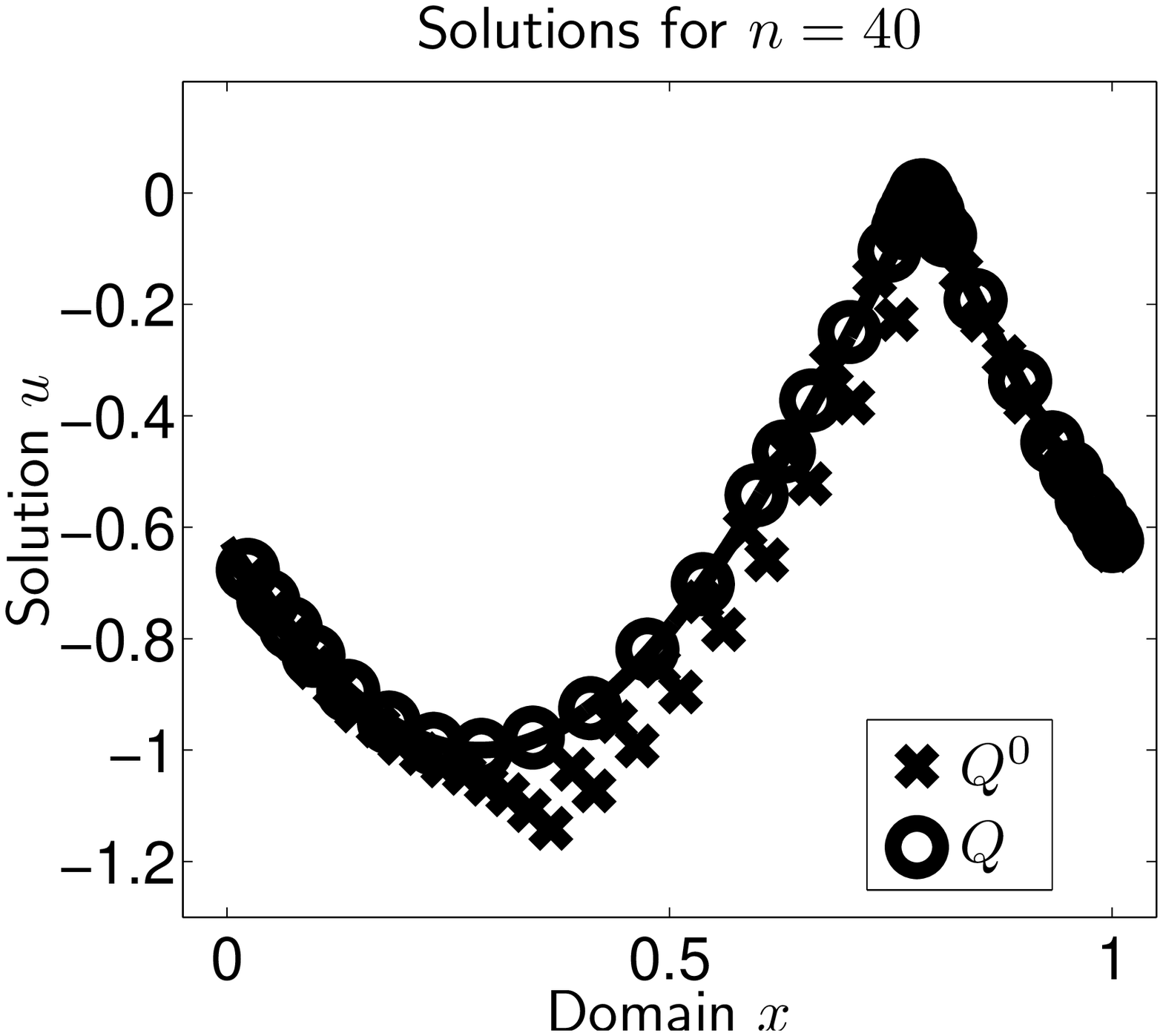}\hfill{}b)\includegraphics[width=0.45\columnwidth]{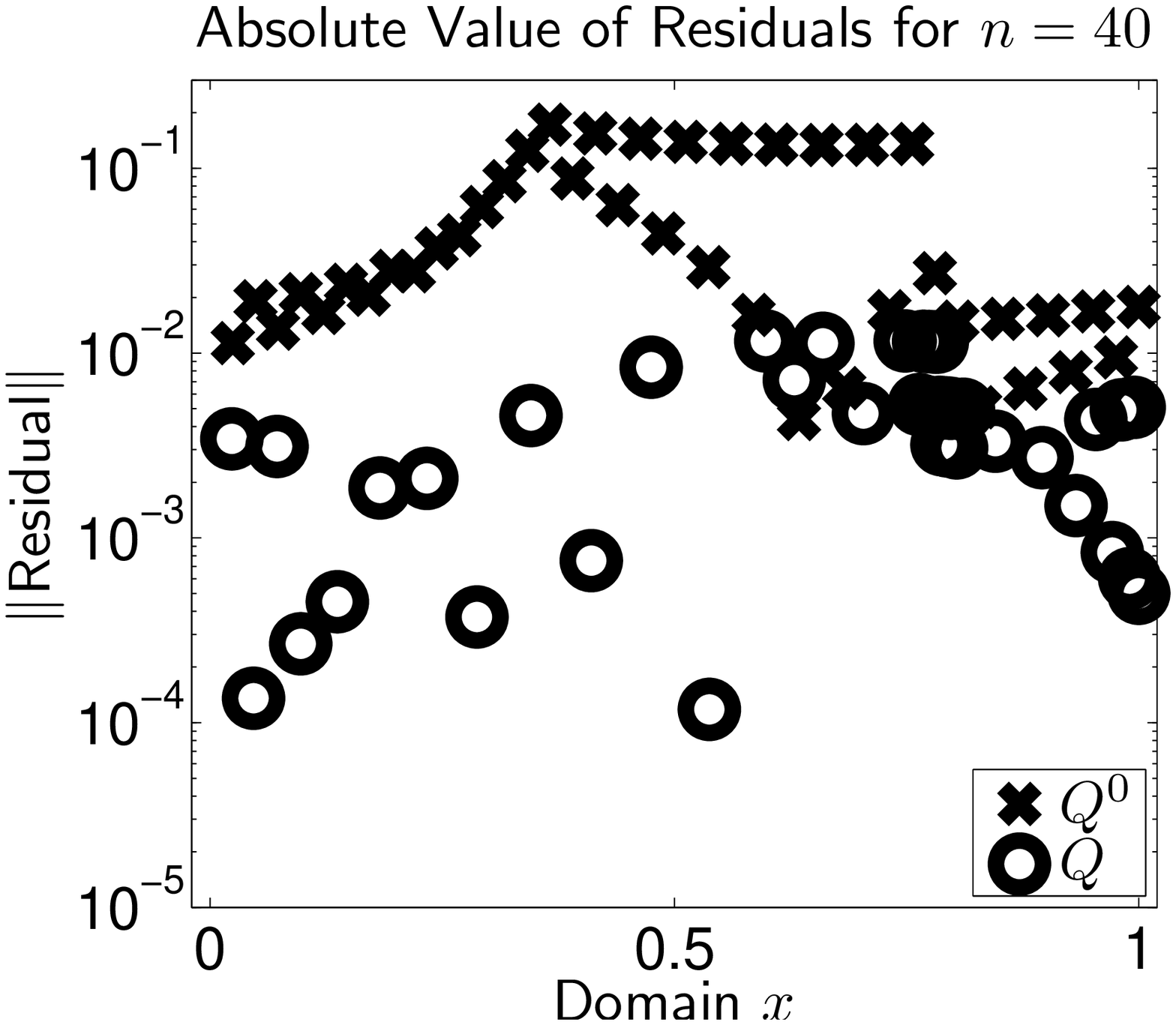}\caption{\label{fig:HamJac-det-soln-resid}a) Solution to the HJ equation \eqref{eq:HamJac eqn}-\eqref{fig:HamJac-comparison-P}
with $n=40$ mesh-points using the computed non-uniform mapping circles
'O' and uniformly spaced mesh-point 'X'. b) The absolute value of
the pointwise residuals for both solutions.}

\end{figure}

\subsection{\label{sub:Oscillatory-Boundary-Value}Oscillatory Boundary Value
Problems}

The following problem was selected from \cite{leveque2007} specifically
because it illustrates a situation for which our approach does not
perform well. Consider the specific Poisson problem \begin{eqnarray}
u''(x) & = & -20+\frac{1}{2}\phi''(x)\cos(\phi(x))-\frac{1}{2}(\phi'(x))^{2}\sin(\phi(x))\label{eq:leveque equation}\\
 &  & u(0)=1\,;\, u(1)=3\label{eq:leveque BC}\end{eqnarray}
 with $a=0.5$. The solution with Dirichlet boundary conditions $u(0)=1$,
$u(1)=3$ is \[
u(x)=1+12x-10x^{2}+a\sin(\phi(x))\,.\]
This analytical solution is depicted for $\phi(x)=5\pi x^{3}$ and
$\phi(x)=20\pi x^{3}$ in Figures \ref{fig:Leveque-solution}a and
\ref{fig:Leveque-solution}b, respectively.%
\begin{figure}
\centering{}a)\includegraphics[width=0.45\columnwidth]{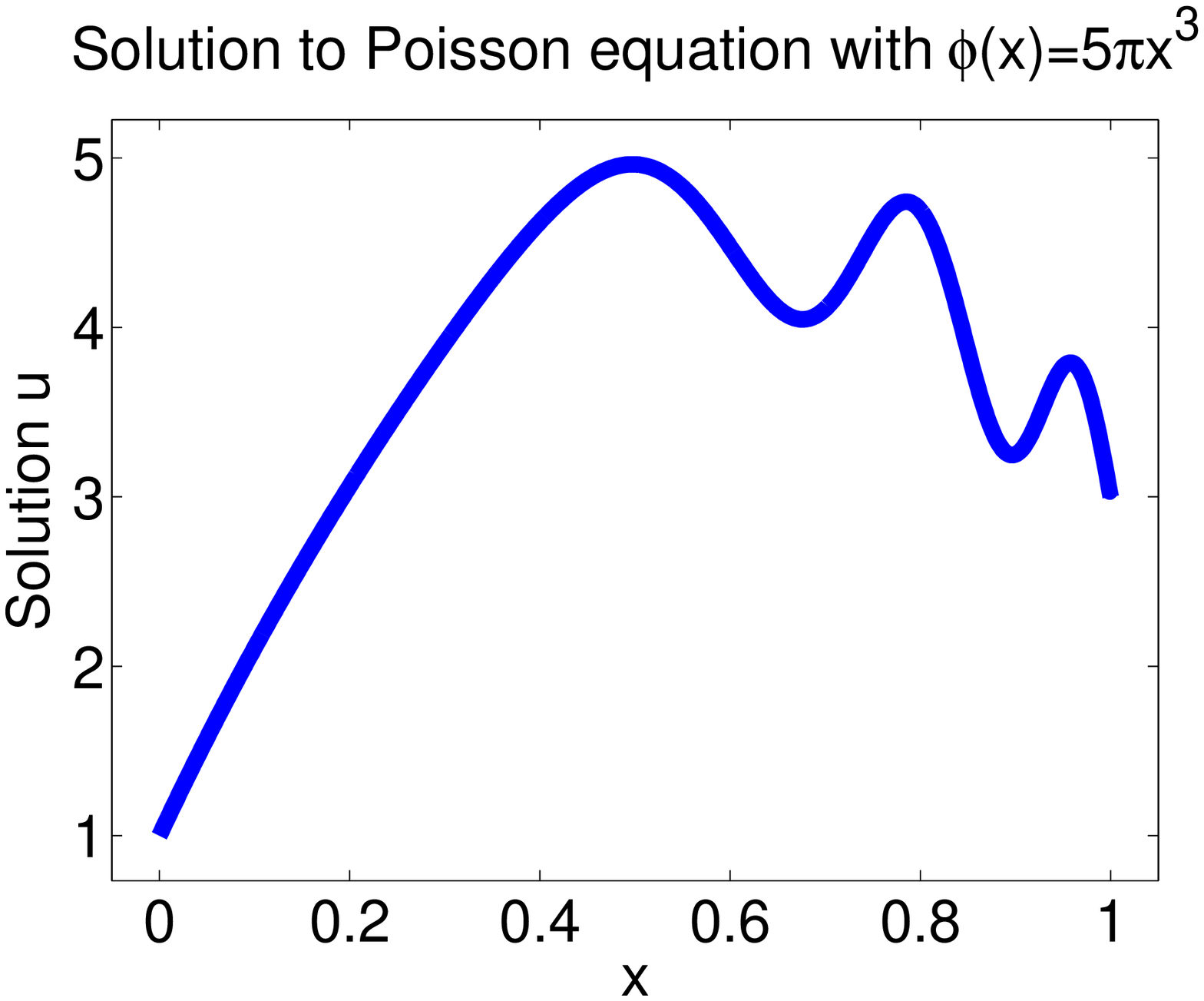}\hfill{}b)\includegraphics[width=0.45\columnwidth]{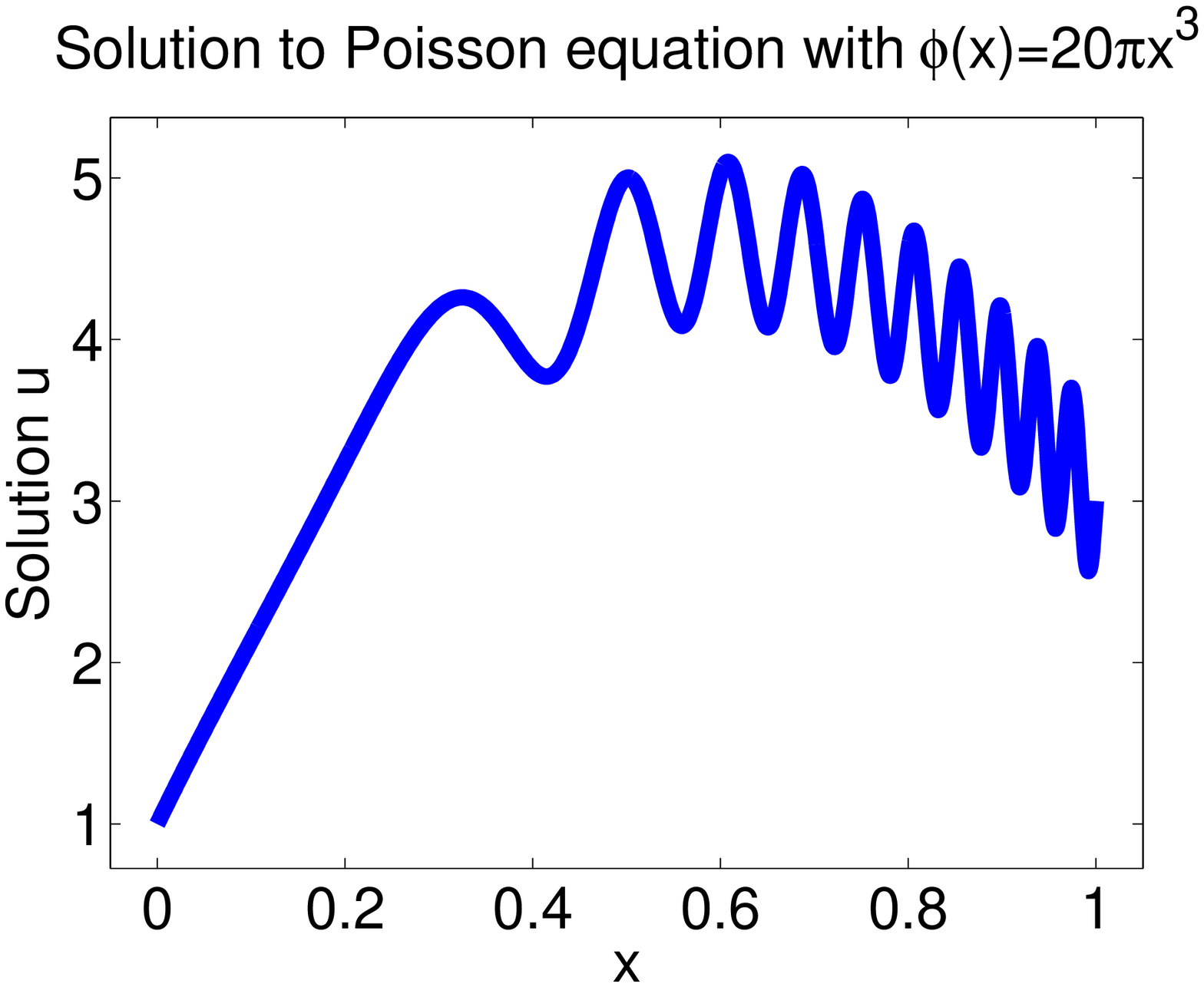}\caption{\label{fig:Leveque-solution}Analytical solution to Poisson equation
from with a) $\phi(x)=5\pi x^{3}$ and b) $\phi(x)=20\pi x^{3}$ from
\cite{leveque2007}.}

\end{figure}

As with the other problems, we first consider $m=15000$ samples of
length $n=10$ and solve \eqref{eq:leveque equation}-\eqref{eq:leveque BC}.
Figure \ref{fig:Leveque-unif-pdf} depicts contour plots of the probability
density of $p(\xi;\mathbb{X}_{(n)}(P))$, where $P$ is a uniform
distribution on $\bar{\mathbf{I}}$, for $\phi(x)=5\pi x^{3}$ and
$\phi(x)=20\pi x^{3}$ in Figures \ref{fig:Leveque-unif-pdf}a and
\ref{fig:Leveque-unif-pdf}b, respectively. We note the relatively
large total variance and that higher frequency forcing oscillations
appear to correlate with larger variance and plan to investigate this
in a future article.%
\begin{figure}
\centering{}a)\includegraphics[width=0.45\columnwidth]{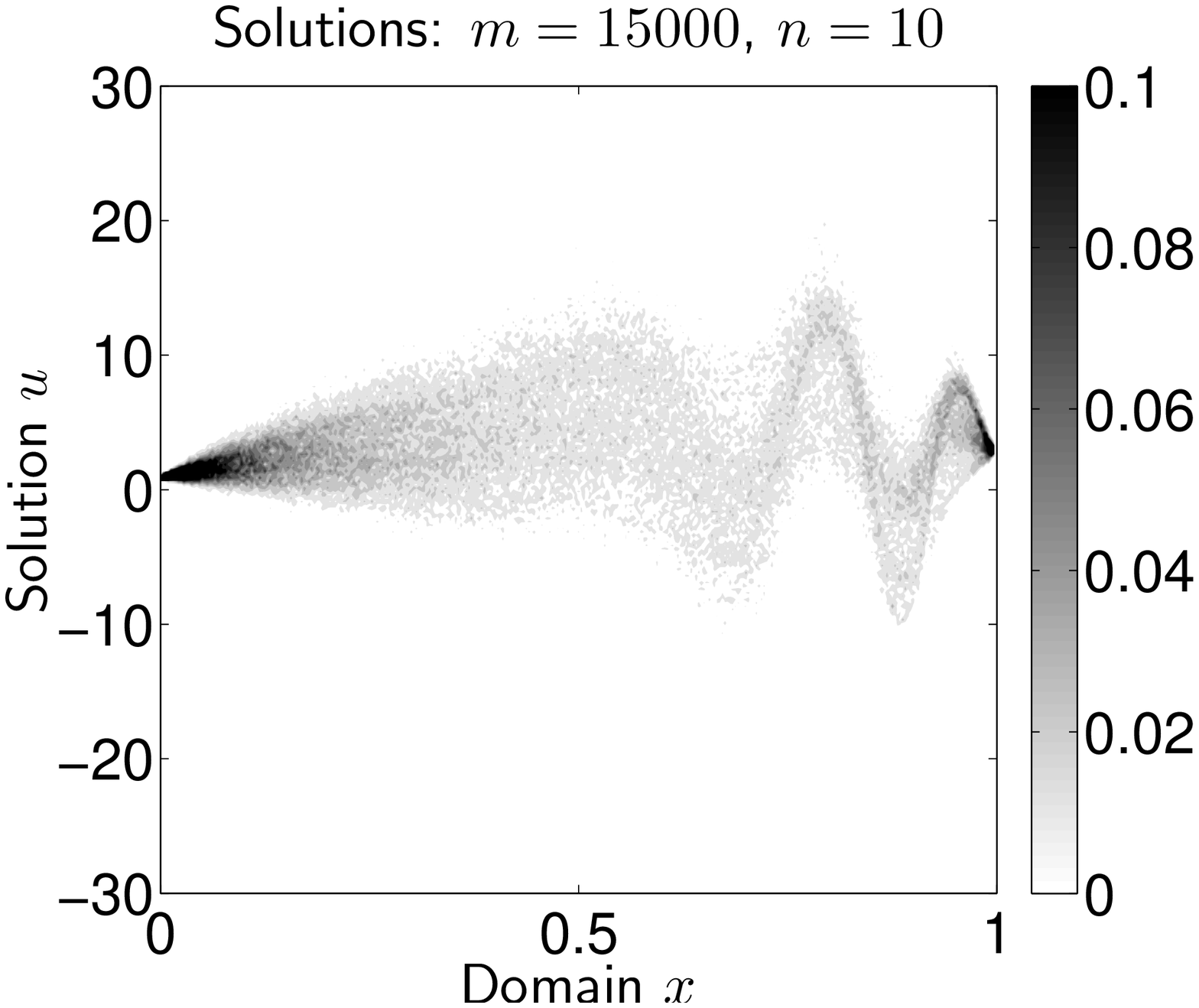}\hfill{}b)\includegraphics[width=0.45\columnwidth]{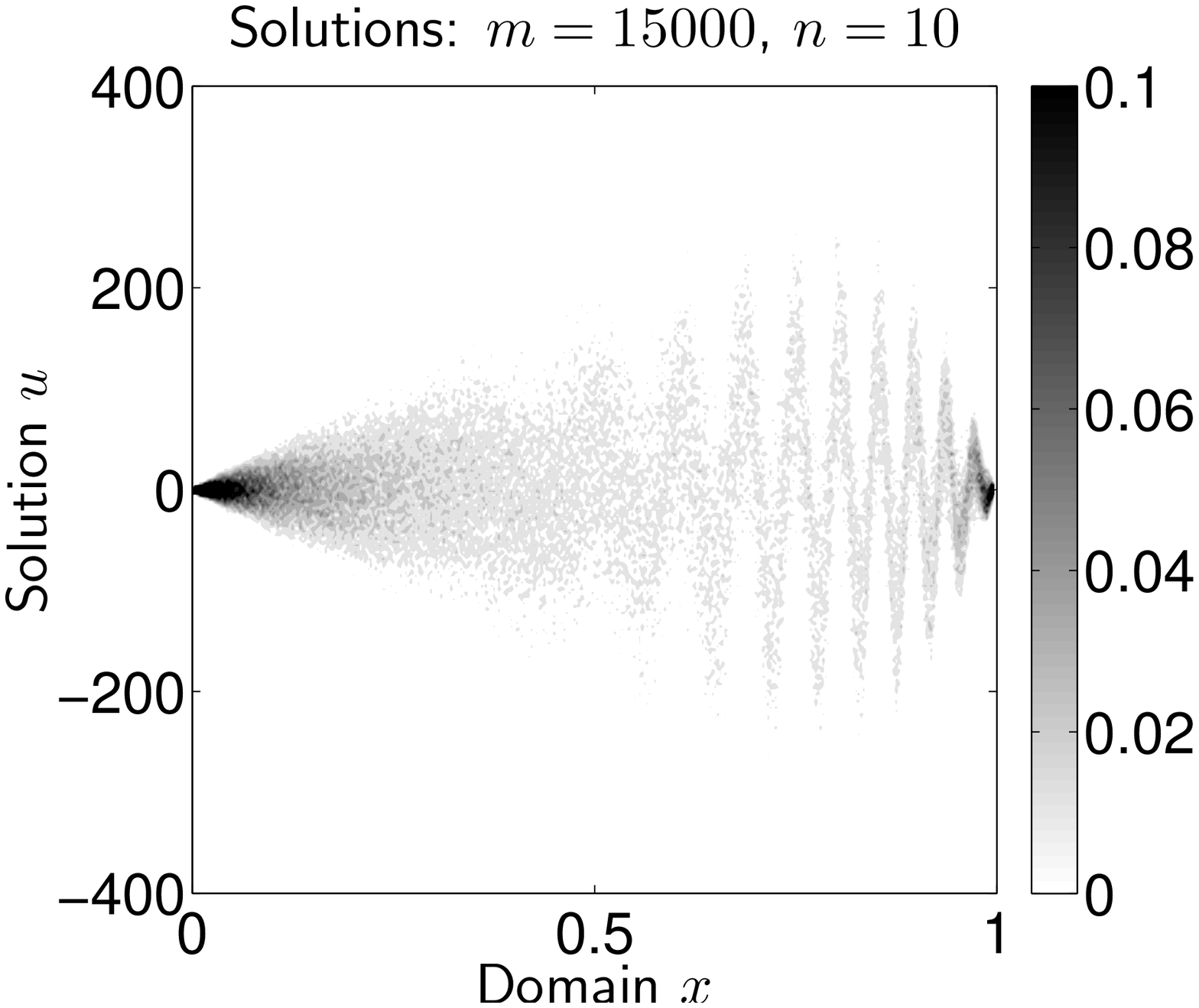}\caption{\label{fig:Leveque-unif-pdf}Contour plot of the probability density
of $p(x,\mathbb{X}_{(n)}(P))$ for $m=15000$ sampled meshes of length
$n=10$ for the Poisson equation defined in \eqref{eq:leveque equation}-\eqref{eq:leveque BC},
where $P$ is a uniform distribution on $\bar{\mathbf{I}}$.}

\end{figure}

Figures \ref{fig:Leveque-comparison-P}a) and \ref{fig:Leveque-comparison-P}b
depict our computed $Q$ for the two $\phi(x)$ functions along with
the uniform mapping using $n=10$. Figure \ref{fig:Leveque-error-conv}
depicts the convergence in error for uniformly and non-uniformly spaced
points as specified by the mappings defined in Figure \ref{fig:Leveque-comparison-P}.
Here, the $Q$ computed by our approach is noticeably worse than using
a uniform spacing. %
\begin{figure}
\centering{}a)\includegraphics[width=0.45\columnwidth]{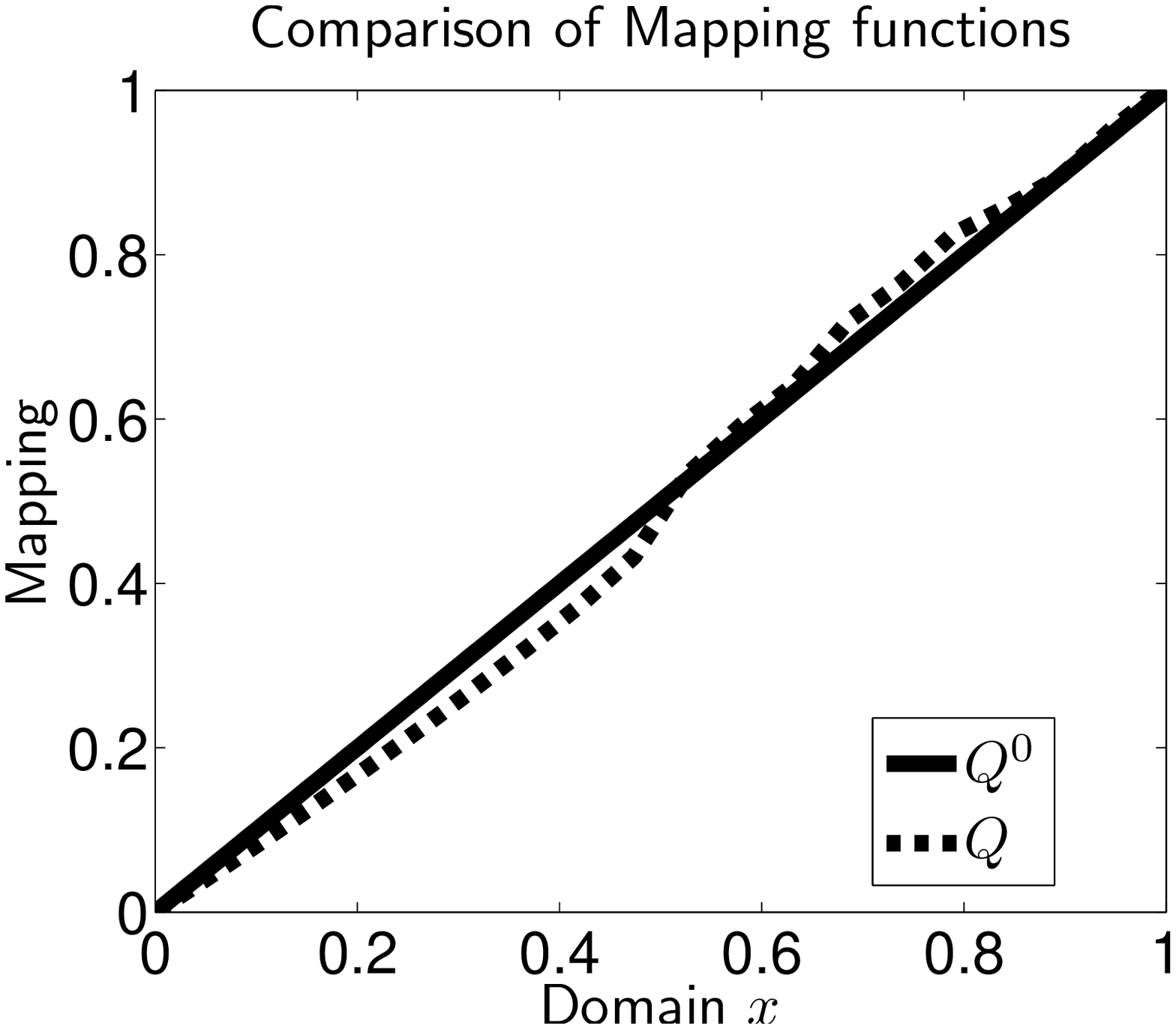}\hfill{}b)\includegraphics[width=0.45\columnwidth]{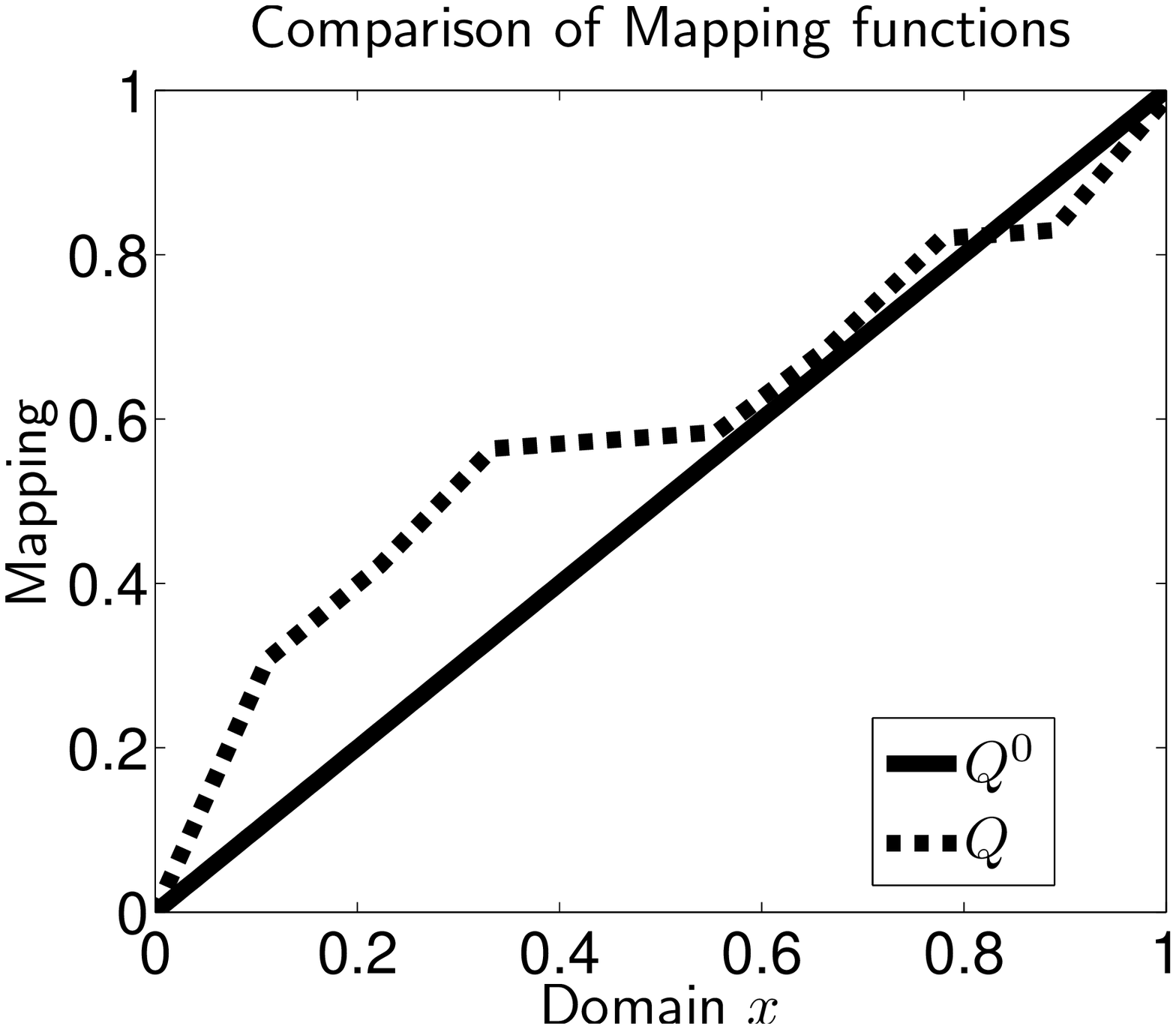}\caption{\label{fig:Leveque-comparison-P}Mapping functions $Q$ and $Q^{0}$
for a) $\phi(x)=5\pi x^{3}$ and b) $\phi(x)=20\pi x^{3}$.}

\end{figure}
\begin{figure}
\centering{}a)\includegraphics[width=0.45\columnwidth]{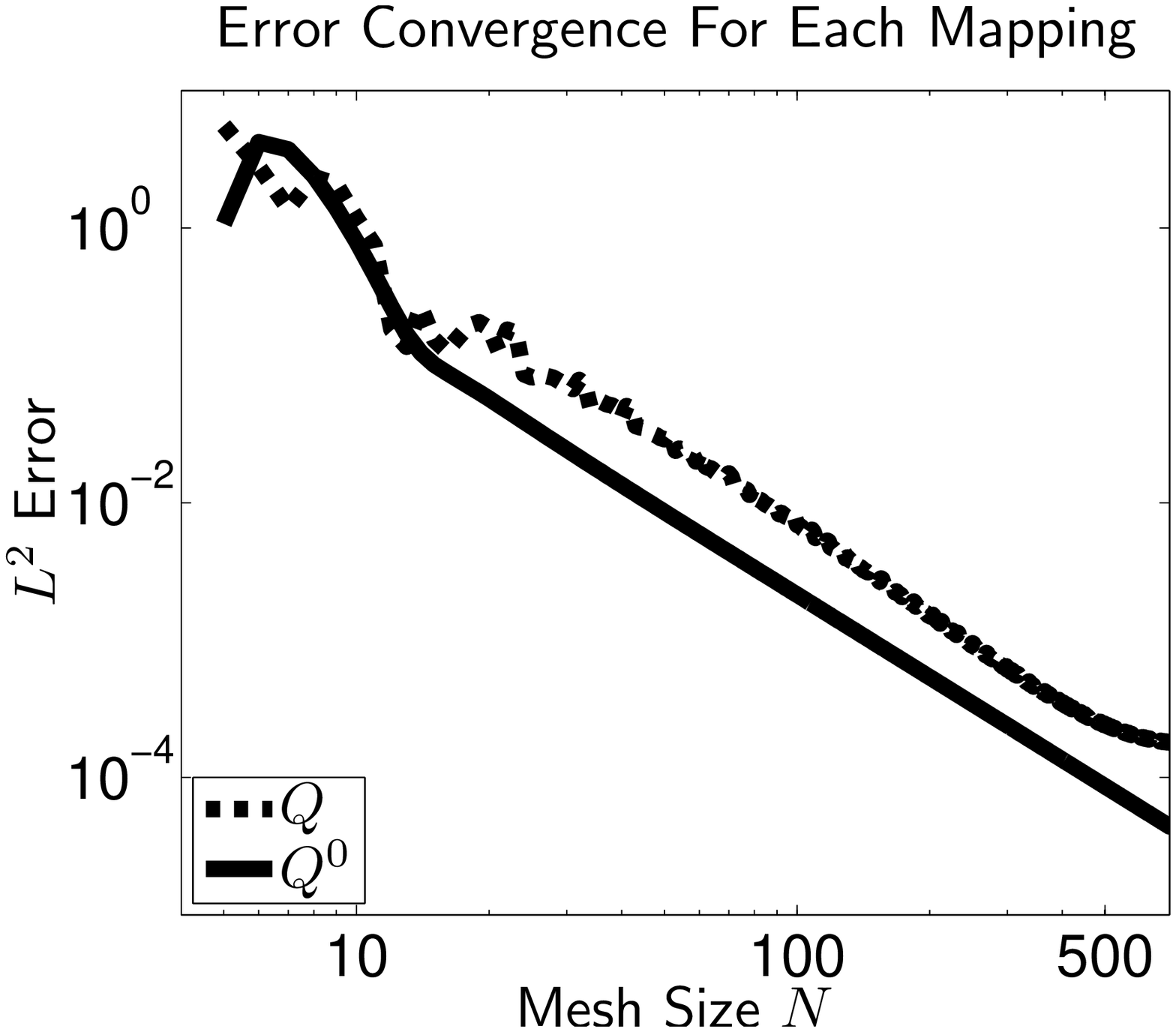}\hfill{}b)\includegraphics[width=0.45\columnwidth]{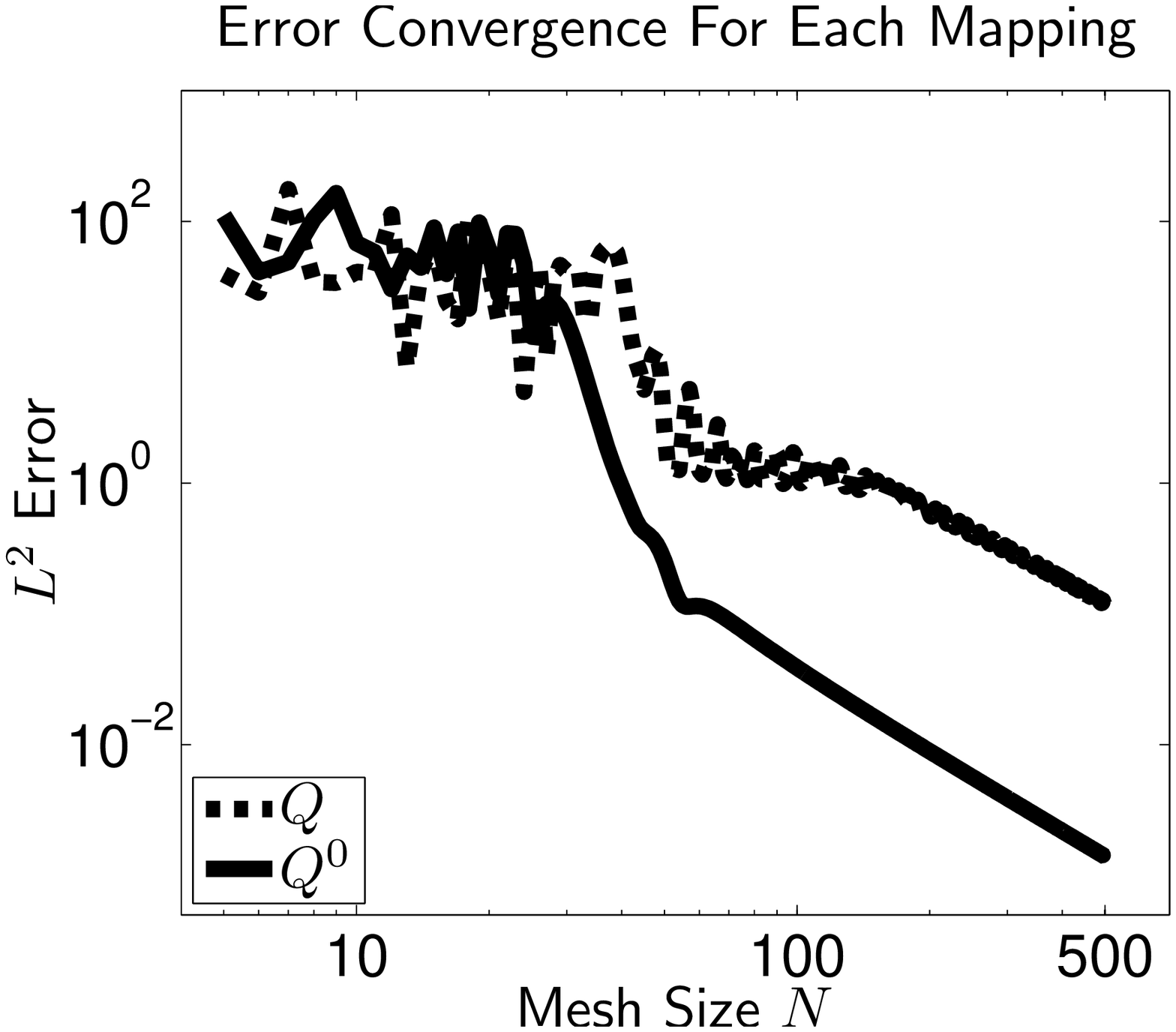}\caption{\label{fig:Leveque-error-conv}Error convergence for uniformly and
non-uniformly spaced points where the non-uniform spacing is determined
by the mapping in Figure \ref{fig:Leveque-comparison-P}.}

\end{figure}

\section{\label{sec:Summary}Summary}

\subsection{Conclusions and Future Work}

In this article, we have proposed a strategy for identifying non-uniform
mappings to improve solution accuracy when using a smaller number
of points. We find that our strategy performs well when faced with
local phenomena, i.e., when there is a subset of the domain which
would benefit from a higher density of mesh-points. With behaviors
which are fundamentally global, such as oscillations, our approach
fares no better than with uniformly placed points.

The immediate goal for the next article is to apply our approach to
higher dimensional systems where we envision the benefit of the $\mathcal{O}(\sqrt{m})$
convergence (independent of $n$) will become even more pronounced.
With the recent release of software to enable calling CUDA (Compute
Unified Device Architecture \cite{cuda}) from Matlab, we also plan
to develop an implementation on a GPGPU. Among the open theoretical
questions we plan to investigate in followup articles, the most pressing
is how to formalize taking the information in the point cloud of sparse-mesh
solutions and construct $Q$. We envision development of a criteria,
which is based in part on the operator discretization strategy and
part on the form of the PDE.

Our algorithm could also greatly benefit from connections to the experimental
design literature. There are criteria designed to aid in choosing
optimal sampling locations, which could provide a more efficient computation
of $Q$ \cite{ad1992}. Our original vision for this work involved
casting $P^{-1}$ as the mapping $Q$ and developing an inverse problem
strategy to identify an optimal $P^{*}$ via optimizing $\bar{v}(P)$.
For this paper, the optimization was not nearly as efficient as simply
computing $Q$. There are, however, stochastic approximation algorithms
(such as Kiefer-Wolfowitz \cite{kw1952amstat}) which are designed
to optimize stochastic functions and there is significant literature
within the control theory community on this topic \cite{ky2003StochasticApproximation}.
This would allow reduction in the number of samples needed to compute
the total variance. We also plan to investigate an improved strategy
for the function approximation choices (such as how many $n$ are
needed). A long-term goal is to investigate how our approach could
be used to update a given set of mesh-points to improve accuracy.
Lastly, there are many opportunities to extend the theoretical analysis.
In particular, the function theoretic basis for the finite element
approach could prove useful in establishing convergence the the appropriate
Sobolev space. Lastly, the rich literature of random matrix theory
will of course improve our knowledge of the distribution of eigenvalues
for the approximate finite difference operator.

\subsection{Discussion}

The widespread availability of massively multi-core architectures
such as General-Purpose Graphics Processing Units (GPGPU's) provides
a unique opportunity for investigating non-traditional and non-conventional
approaches to solving computational problems. It is an opportunity
to completely re-imagine traditional numerical analysis algorithms.
For example, in \cite{co2011jscicomp}, the authors propose a parallel
ODE solver which would be particularly slow on single-core CPU's.
However, when implemented on a GPGPU, it can realize spectacular improvements
in computational efficiency and accuracy. Similarly, the approach
presented above rests on the notion that future trends in architecture
design will produce processors with a large numbers of computational
cores and relatively modest memory resources. Accordingly, we envision
that the implementation of the sampling step on a GPGPU will yield
substantial speedups in efficient mesh generation, since the algorithm
will theoretically scale with the number of computational cores. Like
the parallel ODE solver in \cite{co2011jscicomp}, a single-core implementation
would be completely unacceptable. We assert that algorithms in this
class are beyond embarrassingly parallel; implementing it on anything
but a massively multi-core architecture would be scandalous.

\section{Acknowledgments}

D. M. Bortz and A. J. Christlieb are supported in part by DOD-AFOSR
grant FA9550-09-1-0403. The authors would also like to thank V. Duki\'c
for insightful comments on an earlier draft of this manuscript.

\bibliographystyle{plain}
\bibliography{/home/dmbortz/colorado/bib/string,/home/dmbortz/colorado/bib/article,/home/dmbortz/colorado/bib/book,/home/dmbortz/colorado/bib/incollection,/home/dmbortz/colorado/bib/htbanks,/home/dmbortz/colorado/bib/unpublished,/home/dmbortz/colorado/bib/inproceedings,/home/dmbortz/colorado/bib/dmbortz,/home/dmbortz/colorado/bib/misc,/home/dmbortz/colorado/bib/inbook,/home/dmbortz/colorado/bib/ajchristlieb}

\appendix

\section{Notation}
\begin{lyxlist}{00.00.0000}
\item [{$A_{Q(x^{n})}$}] Matrix generated by a three point stencil (potentially
non-uniform) discretizing the second derivative in a Poisson equation.
The subscript indicates the set of points to use in the discretization.
\item [{$f_{Q(x^{n})}$}] Vector generated by evaluating the forcing function
in a Poisson equation. The subscript indicates the set of points to
use in the discretization.
\item [{$\mathbf{I}$}] Domain for the BVP. In this work, $\mathbf{I}=(0,1)$
and $\bar{\mathbf{I}}=[0,1]$. It is also the probability sample space
for $X$.
\item [{$m$}] Number of sampled meshes.
\item [{$\mu$}] Pointwise mean of $p$, which allows computation of the
expected value of the approximate solutions at a given point $\xi\in\mathbf{I}$,
given that $\xi$ is a point in the mesh.
\item [{$n$}] Number of points in a given mesh.
\item [{$p$}] Function mapping a random vector of length $n$ and a point
$\xi\in\mathbf{I}$ to a random variable. 
\item [{$P$}] Probability distribution.
\item [{$Q$}] Mesh function which maps $\mathbf{I}\to\mathbf{I}$. 
\item [{$\tau_{Q(x^{n})}$}] Taylor Series truncation error from using
a three point stencil to discretize the second derivative on a (potentially
non-uniform) mesh. The subscript indicates the set of point to use
in the discretization.
\item [{$u$}] Analytical solution to a BVP.
\item [{$\mathbf{u}_{0}^{n}$}] $n$-dimensional numerical approximation
to $u$.
\item [{$U$}] Function which maps $\mathbf{I}^{n}\to\mathbb{R}^{n}$ taking
a mesh and solving a given BVP on those mesh points.
\item [{$v$}] Pointwise variance of $p$, which allows computation of
the variance of the approximate solutions at a given point $\xi\in\mathbf{I}$,
given that $\xi$ is a point in the mesh.
\item [{$\bar{v}$}] Average of variance $v$ over $\bar{\mathbf{I}}$.
\item [{$\mathbf{x}_{n}^{0}$}] Mesh of $n$ uniformly spaced points in
$\mathbf{I}$.
\item [{$X(P)$}] Random variable with distribution $P$.
\item [{$\mathbb{X}_{n}$}] Vector of $n$ i.i.d. random variables. $\mathbb{X}_{n}=(X_{1},X_{2},\ldots,X_{n})^{T}$.
\item [{$\mathbb{X}_{(n)}$}] Sorted vector of $n$ i.i.d. random variables.
$\mathbb{X}_{(n)}=(X_{(1)},X_{(2)},\ldots,X_{(n)})^{T}$.
\item [{$\xi$}] Arbitrary point in $\mathbf{I}$.
\end{lyxlist}

\end{document}